\documentclass[12pt,draftcls,onecolumn]{IEEEtran}

\usepackage{extarrows,paralist}                         
\usepackage{mhchem} % Package for chemical equation typesetting
\usepackage{siunitx} % Provides the \SI{}{} command for typesetting SI units
\usepackage{graphicx} % Required for the inclusion of images
\usepackage{amsmath,graphicx}
\usepackage{mathtools}
\DeclarePairedDelimiter\ceil{\lceil}{\rceil}

\usepackage{amsthm}
\newtheorem{lem}{Lemma}
\newtheorem{theorem}{Theorem}

\newtheorem{assump}{Assumption}

\newtheorem{rem}{Remark}
\newtheorem{prop}{Proposition}

\usepackage[square, comma, sort&compress, numbers]{natbib}
\usepackage{mathtools}
\usepackage{graphicx}
\usepackage{float}
\usepackage{caption}
\usepackage{color}
\usepackage{placeins}
\usepackage{float}
\usepackage{tabularx,colortbl}
\usepackage[skip=2pt,font=footnotesize]{caption}
\usepackage{amsmath,subfigure}
\usepackage{listings}
\usepackage{epstopdf}
\usepackage{amssymb}
\usepackage{multirow}
\usepackage{algorithm}
\usepackage{algpseudocode}
\usepackage{hyperref}

\usepackage{enumitem}
\allowdisplaybreaks
\makeatletter
\newcommand{\vast}{\bBigg@{4}}
\newcommand{\Vast}{\bBigg@{5}}
\makeatother

\def\mbb{\mathbb}%R
\def\mb{\mathbf}%vector
\def\mc{\mathcal}%set

\begin{document}
	\title{On the linear convergence of distributed optimization over directed graphs} % Title
	\author{Chenguang Xi, and Usman A. Khan$^\dagger$
		\thanks{
			$^\dagger$C.~Xi and U.~A.~Khan are with the Department of Electrical and Computer Engineering, Tufts University, 161 College Ave, Medford, MA 02155; {\texttt{chenguang.xi@tufts.edu, khan@ece.tufts.edu}}. This work has been partially supported by an NSF Career Award \# CCF-1350264.}
	}

\maketitle
\begin{abstract}
This paper develops a fast distributed algorithm, termed \emph{DEXTRA}, to solve the optimization problem when~$n$ agents reach agreement and collaboratively minimize the sum of their local objective functions over the network, where the communication between the agents is described by a~\emph{directed} graph. Existing algorithms solve the problem restricted to directed graphs with convergence rates of $O(\ln k/\sqrt{k})$ for general convex objective functions and $O(\ln k/k)$ when the objective functions are strongly-convex, where~$k$ is the number of iterations.  We show that, with the appropriate step-size, DEXTRA converges at a linear rate $O(\tau^{k})$ for $0<\tau<1$, given that the objective functions are restricted strongly-convex. The implementation of DEXTRA requires each agent to know its local out-degree. Simulation examples further illustrate our findings.
\end{abstract}


\section{Introduction}\label{s1}
Distributed computation and optimization have gained great interests due to their widespread applications in, e.g., large-scale machine learning,~\cite{ml,distributed_Boyd}, model predictive control,~\cite{distributed_Necoara}, cognitive networks,~\cite{distributed_Mateos,distributed_Bazerque}, source localization,~\cite{distributed_Rabbit,distributed_Khan}, resource scheduling,~\cite{distributed_Chunlin}, and message routing,~\cite{distributed_Neglia}. All of these applications can be reduced to variations of distributed optimization problems by a network of agents when the knowledge of objective functions is distributed over the network. In particular, we consider the problem of minimizing a sum of objectives,~$\sum_{i=1}^{n}f_i(\mb{x})$, where~$f_i:\mbb{R}^p\rightarrow\mbb{R}$ is a private objective function at the $i$th agent of the network.

There are many algorithms to solve the above problem in a distributed manner. A few notable approaches are Distributed Gradient Descent (DGD),~\cite{uc_Nedic,DGD_Yuan}, Distributed Dual Averaging (DDA),~\cite{cc_Duchi}, and the distributed implementations of the Alternating Direction Method of Multipliers (ADMM),~\cite{ADMM_Mota,ADMM_Wei, ADMM_Shi}. The algorithms, DGD and DDA, are essentially gradient-based, where at each iteration a gradient-related step is calculated, followed by averaging over the neighbors in the network. The main advantage of these methods is computational simplicity. However, their convergence rate is slow due to the diminishing step-size, which is required to ensure exact convergence. The convergence rate of DGD and DDA with a diminishing step-size is shown to be~$O(\frac{\ln k}{\sqrt{k}})$,~\cite{uc_Nedic}; under a constant step-size, the algorithm accelerates to~$O(\frac{1}{k})$ at the cost of inexact convergence to a neighborhood of the optimal solution,~\cite{DGD_Yuan}. To overcome such difficulties, some alternate approaches include the Nesterov-based methods, e.g., Distributed Nesterov Gradient (DNG) with a convergence rate of~$O(\frac{\ln k}{k})$, and Distributed Nesterov gradient with Consensus iterations (DNC),~\cite{fast_Gradient}. The algorithm, DNC, can be interpreted to have an inner loop, where information is exchanged, within every outer loop where the optimization-step is performed. The time complexity of the outer loop is~$O(\frac{1}{k^2})$ whereas the inner loop performs a substantial $O(\ln k)$ information exchanges within the~$k$th outer loop. Therefore, the equivalent convergence rate of DNC is $O(\frac{\ln k}{k^2})$. Both DNG and DNC assume the gradient to be bounded and Lipschitz continuous at the same time. The discussion of convergence rate above applies to general convex functions. When the objective functions are further strongly-convex, DGD and DDA have a faster convergence rate of $O(\frac{\ln k}{k})$, and DGD with a constant step-size converges linearly to a neighborhood of the optimal solution. See Table~\ref{table:alg} for a comparison of related algorithms.

Other related algorithms include the distributed implementation of ADMM, based on augmented Lagrangian, where at each iteration the primal and dual variables are solved to minimize a Lagrangian-related function,~\cite{ADMM_Mota,ADMM_Wei, ADMM_Shi}. Comparing to the gradient-based methods with diminishing step-sizes, this type of method converges exactly to the optimal solution with a faster rate of~$O(\frac{1}{k})$ owing to the constant step-size; and further has a linear convergence when the objective functions are strongly-convex. However, the disadvantage is a high computation burden because each agent needs to optimize a subproblem at each iteration. To resolve this issue, Decentralized Linearized ADMM (DLM),~\cite{DLM}, and EXTRA,~\cite{EXTRA}, are proposed, which can be considered as a first-order approximation of decentralized ADMM. DLM and EXTRA converge at a linear rate if the local objective functions are  strongly-convex. All these distributed algorithms,~\cite{uc_Nedic,DGD_Yuan,cc_Duchi,ADMM_Mota,ADMM_Wei, ADMM_Shi,fast_Gradient,NN,DLM,EXTRA}, assume the multi-agent network to be an undirected graph. In contrast, literature concerning directed graphs is relatively limited. The challenge lies in the imbalance caused by the asymmetric information exchange in directed graphs.

We report the papers considering directed graphs here. Broadly, there are three notable approaches, which are all gradient-based algorithms with diminishing step-sizes. The first is called Gradient-Push (GP),~\cite{opdirect_Nedic,opdirect_Tsianous,opdirect_Tsianous2,opdirect_Tsianous3}, which combines gradient-descent and push-sum consensus. The push-sum algorithm,~\cite{ac_directed0,ac_directed}, is first proposed in consensus problems to achieve average-consensus\footnote{See~\cite{c_Jadbabaie,c_Saber,c_Saber2,c_Xiao}, for additional information on average-consensus problems.} in a directed graph, i.e., with a column-stochastic matrix. The idea is based on computing the stationary distribution of the column-stochastic matrix characterized by the underlying multi-agent network and canceling the imbalance by dividing with the right eigenvector of the column-stochastic matrix. Directed-Distributed Gradient Descent (D-DGD),~\cite{D-DGD,D-DPS}, follows the idea of Cai and Ishii's work on average-consensus,~\cite{ac_Cai1}, where a new non-doubly-stochastic matrix is constructed to reach average-consensus. The underlying weighting matrix contains a row-stochastic matrix as well as a column-stochastic matrix, and provides some nice properties similar to doubly-stochastic matrices. In~\cite{opdirect_Makhdoumi}, where we name the method Weight-Balancing-Distributed Gradient Descent (WB-DGD), the authors combine the weight-balancing technique,~\cite{c_Hooi-Tong}, together with gradient-descent. These gradient-based methods,~\cite{opdirect_Nedic,opdirect_Tsianous,opdirect_Tsianous2,opdirect_Tsianous3,D-DGD,D-DPS,opdirect_Makhdoumi}, restricted by the diminishing step-size, converge relatively slow at~$O(\frac{\ln k}{\sqrt{k}})$. Under strongly-convex objective functions, the convergence rate of GP can be accelerated to $O(\frac{\ln k}{k})$,~\cite{opdirect_Nedic3}. We sum up the existing first-order distributed algorithms and provide a comparison in terms of speed, in Table~\ref{table:alg}, including both undirected and directed graphs. In Table~\ref{table:alg},~`I' means DGD with a constant step-size is an Inexact method, and~`C' represents that DADMM has a much higher Computation burden compared to other first-order methods.
\begin{table}[!h]
	\begin{center}
		\begin{tabular}{c|c|c|c|c|}
			\cline{2-4}
			& Algorithms & General Convex & strongly-convex   \\
			\cline{1-4}
			\multicolumn{1}{ |c|  }{\multirow{8}{*}{undirected} } &
			\multicolumn{1}{ l| } {DGD~($\alpha_k$)} & $O(\frac{\ln k}{\sqrt{k}})$ & $O(\frac{\ln k}{k})$   \\ \cline{2-4}
			\multicolumn{1}{ |c|  }{}                        &
			\multicolumn{1}{ l| } {DDA~($\alpha_k$)} & $O(\frac{\ln k}{\sqrt{k}})$ & $O(\frac{\ln k}{k})$   \\ \cline{2-4}
			\multicolumn{1}{ |c|  }{}                        &
			\multicolumn{1}{ l| } {DGD~($\alpha$)   (I)}  & $O(\frac{1}{k})$ &$O(\tau^k)$ \\ \cline{2-4}
			\multicolumn{1}{ |c|  }{}                        &
			\multicolumn{1}{ l| } {DNG~($\alpha_k$)}  & $O(\frac{\ln k}{k})$ &   \\ \cline{2-4}
			\multicolumn{1}{ |c|  }{}                        &
			\multicolumn{1}{ l| } {DNC~($\alpha$)}  & $O(\frac{\ln k}{k^2})$ &  \\ \cline{2-4}
			\multicolumn{1}{ |c|  }{}                        &
			\multicolumn{1}{ l| } {DADMM~($\alpha$)   (C)}  &   & $O(\tau^k)$  \\ \cline{2-4}
			\multicolumn{1}{ |c|  }{}                        &
			\multicolumn{1}{ l| } {DLM~($\alpha$)}  &   & $O(\tau^k)$   \\ \cline{2-4}
			\multicolumn{1}{ |c|  }{}                        &
			\multicolumn{1}{ l| } {EXTRA~($\alpha$)}  & $O(\frac{1}{k})$ & $O(\tau^k)$   \\ \cline{1-4}
			\multicolumn{1}{ |c|  }{\multirow{3}{*}{directed} } &
			\multicolumn{1}{ l| } {GP~($\alpha_k$)} & $O(\frac{\ln k}{\sqrt{k}})$ & $O(\frac{\ln k}{k})$   \\ \cline{2-4}
			\multicolumn{1}{ |c|  }{}                        &
			\multicolumn{1}{ l| } {D-DGD~($\alpha_k$)} & $O(\frac{\ln k}{\sqrt{k}})$ & $O(\frac{\ln k}{k})$   \\ \cline{2-4}
			\multicolumn{1}{ |c|  }{}                        &
			\multicolumn{1}{ l| } {WB-DGD~($\alpha_k$)} & $O(\frac{\ln k}{\sqrt{k}})$ & $O(\frac{\ln k}{k})$   \\ \cline{2-4}
			\multicolumn{1}{ |c|  }{}                        &
			\multicolumn{1}{ l| } {DEXTRA~($\alpha$)}  &  &  $O(\tau^k)$  \\ \cline{1-4}						
	\end{tabular}
	\caption[Table caption text]{Convergence rate of first-order distributed optimization algorithms regarding undirected and directed graphs. }
	\label{table:alg}
\end{center}
\end{table}

In this paper, we propose a fast distributed algorithm, termed DEXTRA, to solve the corresponding distributed optimization problem over directed graphs. We assume that the objective functions are restricted strongly-convex, a relaxed version of strong-convexity, under which we show that DEXTRA converges linearly to the optimal solution of the problem. DEXTRA combines the push-sum protocol and EXTRA. The push-sum protocol has been proven useful in dealing with optimization over digraphs,~\cite{opdirect_Nedic,opdirect_Tsianous,opdirect_Tsianous2,opdirect_Tsianous3}, while EXTRA works well in optimization problems over undirected graphs with a fast convergence rate and a low computation complexity. By integrating the push-sum technique into EXTRA, we show that DEXTRA converges exactly to the optimal solution with a linear rate,~$O(\tau^k)$, when the underlying network is directed. Note that~$O(\tau^k)$ is commonly described as linear and it should be interpreted as linear on a log-scale. The fast convergence rate is guaranteed because DEXTRA has a constant step-size compared with the diminishing step-size used in GP, D-DGD, or WB-DGD. Currently, our formulation is limited to restricted strongly-convex functions. Finally, we note that an earlier version of DEXTRA,~\cite{DEXTRA}, was used in~\cite{the_copy_work} to develop Normalized EXTRAPush. Normalized EXTRAPush implements the DEXTRA iterations after computing the right eigenvector of the underlying column-stochastic, weighting matrix; this computation requires \textit{either} the knowledge of the weighting matrix at each agent, \textit{or}, an iterative algorithm that converges asymptotically to the right eigenvector. Clearly, DEXTRA does not assume such knowledge.

The remainder of the paper is organized as follows. Section~\ref{s2} describes, develops, and interprets the DEXTRA algorithm. Section~\ref{s3} presents the appropriate assumptions and states the main convergence results. In Section~\ref{s4}, we present some lemmas as the basis of the proof of DEXTRA's convergence. The main proof of the convergence rate of DEXTRA is provided in Section~\ref{s5}. We show numerical results in Section~\ref{s6} and Section~\ref{s7} contains the concluding remarks.

\textbf{Notation:} We use lowercase bold letters to denote vectors and uppercase italic letters to denote matrices. We denote by~$[\mb{x}]_i$, the~$i$th component of a vector,~$\mb{x}$. For a matrix,~$A$, we denote by~$[A]_i$, the~$i$th row of $A$, and by~$[A]_{ij}$, the $(i,j)$th element of $A$. The matrix,~$I_n$, represents the~$n\times n$ identity, and~$\mb{1}_n$ and $\mb{0}_n$ are the~$n$-dimensional vector of all $1$'s and $0$'s. The inner product of two vectors,~$\mb{x}$ and~$\mb{y}$, is~$\langle\mb{x},\mb{y}\rangle$. The Euclidean norm of any vector, $\mb{x}$, is denoted by $\|\mb{x}\|$. We define the $A$-matrix norm,~$\left\|\mb{x}\right\|_A^2$, of any vector,~$\mb{x}$, as~$$\left\|\mb{x}\right\|_A^2\triangleq\langle\mb{x},A\mb{x}\rangle=\langle\mb{x},A^\top\mb{x}\rangle=\langle\mb{x},\frac{A+A^\top}{2}\mb{x}\rangle,$$where $A$ is not necessarily symmetric. Note that the $A$-matrix norm is non-negative only when~$A+A^\top$ is Positive Semi-Definite (PSD). If a symmetric matrix,~$A$, is PSD, we write~$A\succeq 0$, while~$A\succ 0$ means~$A$ is Positive Definite (PD). The largest and smallest eigenvalues of a matrix~$A$ are denoted as~$\lambda_{\max}(A)$ and~$\lambda_{\min}(A)$. The smallest~\emph{nonzero} eigenvalue of a matrix~$A$ is denoted as~$\widetilde{\lambda}_{\min}(A)$. For any~$f(\mb{x})$,~$\nabla f(\mb{x})$ denotes the gradient of~$f$ at~$\mb{x}$.

\section{DEXTRA Development}\label{s2}
In this section, we formulate the optimization problem and describe DEXTRA. We first derive an informal but intuitive proof showing that DEXTRA pushes the agents to achieve consensus and reach the optimal solution. The EXTRA algorithm,~\cite{EXTRA}, is briefly recapitulated in this section. We derive DEXTRA to a similar form as EXTRA such that our algorithm can be viewed as an improvement of EXTRA suited to the case of directed graphs. This reveals the meaning behind DEXTRA: Directed EXTRA. Formal convergence results and proofs are left to Sections~\ref{s3},~\ref{s4}, and~\ref{s5}.

Consider a strongly-connected network of~$n$ agents communicating over a directed graph,~$\mc{G}=(\mc{V},\mc{E})$, where~$\mc{V}$ is the set of agents, and~$\mc{E}$ is the collection of ordered pairs,~$(i,j),i,j\in\mc{V}$, such that agent~$j$ can send information to agent~$i$. Define~$\mc{N}_i^{{\scriptsize \mbox{in}}}$ to be the collection of in-neighbors, i.e., the set of agents that can send information to agent~$i$. Similarly,~$\mc{N}_i^{{\scriptsize \mbox{out}}}$ is the set of out-neighbors of agent~$i$. We allow both~$\mc{N}_i^{{\scriptsize \mbox{in}}}$ and~$\mc{N}_i^{{\scriptsize \mbox{out}}}$ to include the node~$i$ itself. Note that in a directed graph when~$(i,j)\in\mc{E}$, it is not necessary that~$(j,i)\in\mc{E}$. Consequently,~$\mc{N}_i^{{\scriptsize \mbox{in}}}\neq\mc{N}_i^{{\scriptsize \mbox{out}}}$, in general. We focus on solving a convex optimization problem that is distributed over the above multi-agent network. In particular, the network of agents cooperatively solve the following optimization problem:
\begin{align}
\mbox{P1}:
\quad\mbox{min  }&f(\mb{x})=\sum_{i=1}^nf_i(\mb{x}),\nonumber
\end{align}
where each local objective function, $f_i:\mbb{R}^p\rightarrow\mbb{R}$, is convex and differentiable, and known only by agent $i$. Our goal is to develop a distributed iterative algorithm such that each agent converges to the global solution of Problem P1.

\subsection{EXTRA for undirected graphs}
EXTRA is a fast exact first-order algorithm that solve Problem P1 when the communication network is undirected. At the~$k$th iteration, agent~$i$ performs the following update:
\begin{align}\label{extra}
\mb{x}_i^{k+1}=&\mb{x}_i^k+\sum_{j\in\mc{N}_i}w_{ij}\mb{x}_j^k-\sum_{j\in\mc{N}_i}\widetilde{w}_{ij}\mb{x}_j^{k-1}-\alpha\left[\nabla f_i(\mb{x}_i^k)-\nabla f_i(\mb{x}_i^{k-1})\right],
\end{align}
where the weights, $w_{ij}$, form a weighting matrix, $W=\left\{w_{ij}\right\}$, that is symmetric and doubly-stochastic. The collection~$\widetilde{W}=\left\{\widetilde{w}_{ij}\right\}$ satisfies~$\widetilde{W}=\theta I_n+(1-\theta)W$, with some~$\theta\in(0,\frac{1}{2}]$. The update in Eq.~\eqref{extra} converges to the optimal solution at each agent~$i$ with a convergence rate of~$O(\frac{1}{k})$ and converges linearly when the objective functions are strongly-convex. To better represent EXTRA and later compare with DEXTRA, we write Eq.~\eqref{extra} in a matrix form. Let~$\mb{x}^k$,~$\nabla\mb{f}(\mb{x}^k)\in\mbb{R}^{np}$ be the collections of all agent states and gradients at time~$k$, i.e.,~$\mb{x}^k\triangleq[\mb{x}_1^k;\cdots;\mb{x}_n^k]$,~$\nabla\mb{f}(\mb{x}^k)\triangleq[\nabla f_1(\mb{x}_1^k);\cdots;\nabla f_n(\mb{x}_n^k)]$, and~$W$,~$\widetilde{W}\in\mbb{R}^{n\times n}$ be the weighting matrices collecting weights,~$w_{ij}$,~$\widetilde{w}_{ij}$, respectively. Then, Eq.~\eqref{extra} can be represented in a matrix form as:
\begin{align}\label{extra_matrix}
\mb{x}^{k+1}=&\left[(I_n+W)\otimes I_p\right]\mb{x}^k-(\widetilde{W}\otimes I_p)\mb{x}^{k-1}-\alpha\left[\nabla\mb{f}(\mb{x}^k)-\nabla\mb{f}(\mb{x}^{k-1})\right],
\end{align}
where the symbol~$\otimes$ is the Kronecker product. We now state DEXTRA and derive it in a similar form as EXTRA.
\subsection{DEXTRA Algorithm}
To solve the Problem P1 suited to the case of directed graphs, we propose DEXTRA that can be described as follows. Each agent,~$j\in\mc{V}$, maintains two vector variables:~$\mb{x}_j^k$,~$\mb{z}_j^k\in\mbb{R}^p$, as well as a scalar variable,~$y_j^k\in\mbb{R}$, where~$k$ is the discrete-time index. At the~$k$th iteration, agent~$j$ weights its states,~$a_{ij}\mb{x}_j^k$,~$a_{ij}y_j^k$, as well as~$\widetilde{a}_{ij}\mb{x}_j^{k-1}$, and sends these to each of its out-neighbors,~$i\in\mc{N}_j^{{\scriptsize \mbox{out}}}$, where the weights,~$a_{ij}$, and,~$\widetilde{a}_{ij}$,'s are such that:
\begin{align}
a_{ij}&=\left\{
\begin{array}{rl}
>0,&i\in\mc{N}_j^{{\scriptsize \mbox{out}}},\\
0,&\mbox{otw.},
\end{array}
\right.
\quad
\sum_{i=1}^na_{ij}=1,\forall j,\label{a}\\
\widetilde{a}_{ij}&=\left\{
\begin{array}{rl}
\theta+(1-\theta)a_{ij},&i=j,\\
(1-\theta)a_{ij},&i\neq j,
\end{array}
\right.
\qquad\quad
\forall j,\label{aa}
\end{align}
where~$\theta\in(0,\frac{1}{2}]$. With agent~$i$ receiving the information from its in-neighbors,~$j\in\mc{N}_i^{{\scriptsize \mbox{in}}}$, it calculates the state,~$\mb{z}_i^k$, by dividing~$\mb{x}_i^k$ over~$y_i^k$, and updates~$\mb{x}_i^{k+1}$ and~$y_i^{k+1}$  as follows:
\begin{subequations}\label{alg1}
\begin{align}
\mb{z}_i^k=&\frac{\mb{x}_i^k}{y_i^k},\label{alg1a}\\
\mb{x}_i^{k+1}=&\mb{x}_i^k+\sum_{j\in\mc{N}_i^{{\tiny \mbox{in}}}}\left(a_{ij}\mb{x}_j^k\right)-\sum_{j\in\mc{N}_i^{{\tiny \mbox{in}}}}\left(\widetilde{a}_{ij}\mb{x}_j^{k-1}\right)-\alpha\left[\nabla f_i(\mb{z}_i^k)-\nabla f_i(\mb{z}_i^{k-1})\right],\label{alg1b}\\
y_i^{k+1}=&\sum_{j\in\mc{N}_i^{{\tiny \mbox{in}}}}\left(a_{ij}y_j^k\right).\label{alg1c}
\end{align}
\end{subequations}
In the above,~$\nabla f_i(\mb{z}_i^k)$ is the gradient of the function~$f_i(\mb{z})$ at~$\mb{z}=\mb{z}_i^k$, and~$\nabla f_i(\mb{z}_i^{k-1})$ is the gradient at~$\mb{z}_i^{k-1}$, respectively. The method is initiated with an arbitrary vector,~$\mb{x}_i^0$, and with~$y_i^0=1$ for any agent~$i$. The step-size,~$\alpha$, is a positive number within a certain interval. We will explicitly discuss the range of~$\alpha$ in Section~\ref{s3}. We adopt the convention that~$\mb{x}_i^{-1}=\mb{0}_p$ and $\nabla f_i(\mb{z}_i^{-1})=\mb{0}_p$, for any agent $i$, such that at the first iteration,~i.e.,~$k=0$, we have the following iteration instead of Eq.~\eqref{alg1}, 
\begin{subequations}\label{alg1_0}
\begin{align}
\mb{z}_i^0=&\frac{\mb{x}_i^0}{y_i^0},\\
\mb{x}_i^{1}=&\sum_{j\in\mc{N}_i^{{\tiny \mbox{in}}}}\left(a_{ij}\mb{x}_j^0\right)-\alpha\nabla f_i(\mb{z}_i^0),\\
y_i^{1}=&\sum_{j\in\mc{N}_i^{{\tiny \mbox{in}}}}\left(a_{ij}y_j^0\right).
\end{align}
\end{subequations}
We note that the implementation of Eq.~\eqref{alg1} needs each agent to have the knowledge of its out-neighbors (such that it can design the weights according to Eqs.~\eqref{a} and~\eqref{aa}). In a more restricted setting, e.g., a broadcast application where it may not be possible to know the out-neighbors, we may use~$a_{ij}=|\mc{N}_j^{{\scriptsize \mbox{out}}}|^{-1}$; thus, the implementation only requires each agent to know its out-degree,~\cite{opdirect_Nedic,opdirect_Tsianous,opdirect_Tsianous2,opdirect_Tsianous3,D-DGD,D-DPS,opdirect_Makhdoumi}.

To simplify  the proof, we write DEXTRA, Eq.~\eqref{alg1}, in a matrix form. Let,~$A=\{a_{ij}\}\in\mbb{R}^{n\times n}$,~$\widetilde{A}=\{\widetilde{a}_{ij}\}\in\mbb{R}^{n\times n}$, be the collection of weights,~$a_{ij}$,~$\widetilde{a}_{ij}$, respectively. It is clear that both $A$ and $\widetilde{A}$ are column-stochastic matrices. Let $\mb{x}^k$,~$\mb{z}^k$,~$\nabla\mb{f}(\mb{x}^k)\in\mbb{R}^{np}$, be the collection of all agent states and gradients at time~$k$, i.e.,~$\mb{x}^k\triangleq[\mb{x}_1^k;\cdots;\mb{x}_n^k]$,~$\mb{z}^k\triangleq[\mb{z}_1^k;\cdots;\mb{z}_n^k]$,~$\nabla\mb{f}(\mb{x}^k)\triangleq[\nabla f_1(\mb{x}_1^k);\cdots;\nabla f_n(\mb{x}_n^k)]$, and~$\mb{y}^k\in\mbb{R}^n$ be the collection of agent states,~$y_i^k$, i.e.,~$\mb{y}^k\triangleq[y_1^k;\cdots;y_n^k]$. Note that at time~$k$, ~$\mb{y}^k$ can be represented by the initial value,~$\mb{y}^0$:
\begin{align}
\mb{y}^k=A\mb{y}^{k-1}=A^k\mb{y}^0=A^k\cdot\mb{1}_n.
\end{align}
Define a diagonal matrix,~$D^k\in\mbb{R}^{n\times n}$, for each~$k$, such that the~$i$th element of~$D^k$ is~$y_i^k$, i.e.,
\begin{equation}\label{D}
D^k=\mbox{diag}\left(\mb{y}^k\right)=\mbox{diag}\left(A^k\cdot\mb{1}_n\right).
\end{equation}
Given that the graph, $\mc{G}$, is strongly-connected and the corresponding weighting matrix, $A$, is non-negative, it follows that~$D^k$ is invertible for any~$k$. Then, we can write Eq.~\eqref{alg1} in the matrix form equivalently as follows:
\begin{subequations}\label{alg1_matrix}
	\begin{align}
	\mb{z}^k=&\left(\left[D^k\right]^{-1}\otimes I_p\right)\mb{x}^k,\label{alg1a_matrix}\\
	\mb{x}^{k+1}=&\mb{x}^k+(A\otimes I_p)\mb{x}^k-(\widetilde{A}\otimes I_p)\mb{x}^{k-1}-\alpha\left[\nabla\mb{f}(\mb{z}^k)-\nabla\mb{f}(\mb{z}^{k-1})\right],\label{alg1b_matrix}\\
	\mb{y}^{k+1}=&A\mb{y}^k,\label{alg1c_matrix}
	\end{align}
\end{subequations}
where both of the weight matrices,~$A$ and~$\widetilde{A}$, are column-stochastic and satisfy the relationship:~$\widetilde{A}=\theta I_n+(1-\theta)A$ with some~$\theta\in(0,\frac{1}{2}]$. From Eq.~\eqref{alg1a_matrix}, we obtain for any $k$
\begin{align}
\mb{x}^k=&\left(D^k\otimes I_p\right)\mb{z}^k.
\end{align}
Therefore, Eq.~\eqref{alg1_matrix} can be represented as a single equation:
\begin{align}\label{alg2}
&\left(D^{k+1}\otimes I_p\right)\mb{z}^{k+1}=\left[(I_n+A)D^{k}\otimes I_p\right]\mb{z}^{k}-(\widetilde{A}D^{k-1}\otimes I_p)\mb{z}^{k-1}-\alpha\left[\nabla\mb{f}(\mb{z}^k)-\nabla\mb{f}(\mb{z}^{k-1})\right].
\end{align}
We refer to the above algorithm as DEXTRA, since Eq.~\eqref{alg2} has a similar form as EXTRA in Eq.~\eqref{extra_matrix} and is designed to solve Problem P1 in the case of directed graphs. We state our main result in Section~\ref{s3}, showing that as time goes to infinity, the iteration in Eq.~\eqref{alg2} pushes~$\mb{z}^k$ to achieve consensus and reach the optimal solution in a linear rate. Our proof in this paper will based on the form, Eq.~\eqref{alg2}, of DEXTRA.

\subsection{Interpretation of DEXTRA}
In this section, we give an intuitive interpretation on DEXTRA's convergence to the optimal solution; the formal proof will appear in Sections~\ref{s4} and~\ref{s5}. Since $A$ is column-stochastic, the sequence, $\left\{\mb{y}^k\right\}$, generated by Eq. \eqref{alg1c_matrix}, satisfies $\lim_{k\rightarrow\infty}\mb{y}^k=\boldsymbol{\pi}$, where $\boldsymbol{\pi}$ is some vector in the span of $A$'s right-eigenvector corresponding to the eigenvalue~$1$. We also obtain that $D^\infty=\mbox{diag}\left(\boldsymbol{\pi}\right)$. For the sake of argument, let us assume that the sequences,~$\left\{\mb{z}^k\right\}$ and~$\left\{\mb{x}^k\right\}$, generated by DEXTRA, Eq.~\eqref{alg1_matrix} or~\eqref{alg2}, also converge to their own limits,~$\mb{z}^\infty$ and~$\mb{x}^\infty$, respectively (not necessarily true). According to the updating rule in Eq.~\eqref{alg1b_matrix}, the limit~$\mb{x}^\infty$ satisfies
\begin{align}
\mb{x}^\infty=&\mb{x}^\infty+(A\otimes I_p)\mb{x}^\infty-(\widetilde{A}\otimes I_p)\mb{x}^\infty-\alpha\left[\nabla\mb{f}(\mb{z}^\infty)-\nabla\mb{f}(\mb{z}^\infty)\right],
\end{align}
which implies that~$[(A-\widetilde{A})\otimes I_p]\mb{x}^\infty=\mb{0}_{np}$, or~$\mb{x}^\infty=\boldsymbol{\pi}\otimes\mb{u}$ for some vector,~$\mb{u}\in\mbb{R}^p$. It follows from Eq.~\eqref{alg1a_matrix} that
\begin{align}
\mb{z}^\infty=&\left(\left[D^\infty\right]^{-1}\otimes I_p\right)\left(\boldsymbol{\pi}\otimes\mb{u}\right)=\mb{1}_n\otimes\mb{u},
\end{align}
where the consensus is reached. The above analysis reveals the idea of DEXTRA, which is to overcome the imbalance of agent states occurred when the graph is directed: both~$\mb{x}^\infty$ and~$\mb{y}^\infty$ lie in the span of~$\boldsymbol{\pi}$; by dividing~$\mb{x}^\infty$ over~$\mb{y}^\infty$, the imbalance is canceled.

Summing up the updates in Eq.~\eqref{alg1b_matrix} over~$k$ from $0$ to $\infty$, we obtain that
\begin{align}\nonumber
\mb{x}^{\infty}=\left(A\otimes I_p\right)\mb{x}^\infty-\alpha\nabla\mb{f}(\mb{z}^\infty)-\sum_{r=0}^{\infty}\left[\left(\widetilde{A}-A\right)\otimes I_p\right]\mb{x}^r;
\end{align}
note that the first iteration is slightly different as shown in Eqs.~\eqref{alg1_0}. Consider~$\mb{x}^\infty=\boldsymbol{\pi}\otimes\mb{u}$ and the preceding relation. It follows that the limit,~$\mb{z}^\infty$, satisfies
\begin{align}
\alpha\nabla\mb{f}(\mb{z}^\infty)=-\sum_{r=0}^{\infty}\left[\left(\widetilde{A}-A\right)\otimes I_p\right]\mb{x}^r.
\end{align}
Therefore, we obtain that
\begin{align}
\alpha\left(\mb{1}_n\otimes I_p\right)^\top\nabla\mb{f}(\mb{z}^\infty)=-\left[\mb{1}_n^\top\left(\widetilde{A}-A\right)\otimes I_p\right]\sum_{r=0}^{\infty}\mb{x}^r=\mb{0}_p,\nonumber
\end{align}
which is the optimality condition of Problem P1 considering that $\mb{z}^\infty=\mb{1}_n\otimes\mb{u}$. Therefore, given the assumption that the sequence of DEXTRA iterates,~$\left\{\mb{z}^k\right\}$ and~$\left\{\mb{x}^k\right\}$, have limits,~$\mb{z}^\infty$ and~$\mb{x}^\infty$, we have the fact that~$\mb{z}^\infty$ achieves consensus and reaches the optimal solution of Problem P1. In the next section, we state our main result of this paper and we defer the formal proof to Sections~\ref{s4} and~\ref{s5}.

\section{Assumptions and Main Results}~\label{s3}
With appropriate assumptions, our main result states that DEXTRA converges to the optimal solution of Problem P1 linearly. In this paper, we assume that the agent graph,~$\mc{G}$, is strongly-connected; each local function,~$f_i:\mbb{R}^p\rightarrow\mbb{R}$, is convex and differentiable, and the optimal solution of Problem P1 and the corresponding optimal value exist. Formally, we denote the optimal solution by~$\mb{u}\in\mbb{R}^p$ and optimal value by~$f^*$, i.e.,
\begin{align}
f^*=f(\mb{u})=\min_{\mb{x}\in\mbb{R}^p}f(\mb{x}).
\end{align}
Let~$\mb{z}^*\in\mbb{R}^{np}$ be defined as
\begin{align}\label{z*}
\mb{z}^*=\mb{1}_n\otimes\mb{u}.
\end{align}
Besides the above assumptions, we emphasize some other assumptions regarding the objective functions and weighting matrices, which are formally presented as follows.
\begin{assump}[Functions and Gradients]\label{asp1}
	Each private function,~$f_i$, is convex and differentiable and satisfies the following assumptions.
	\begin{enumerate}[label=(\alph*)]
		\item The function,~$f_i$, has Lipschitz gradient with the constant~$L_{f_i}$, i.e.,~$\|\nabla f_i(\mb{x})-\nabla f_i(\mb{y})\|\leq L_{f_i}\|\mb{x}-\mb{y}\|$,~$\forall\mb{x}, \mb{y}\in\mbb{R}^p$.
		\item The function,~$f_i$,  is restricted strongly-convex\footnote{There are different definitions of restricted strong-convexity. We use the same as the one used in EXTRA,~\cite{EXTRA}.} with respect to point $\mb{u}$ with a positive constant~$S_{f_i}$, i.e.,~$S_{f_i}\|\mb{x}-\mb{u}\|^2\leq\langle\nabla f_i(\mb{x})-\nabla f_i(\mb{u}),\mb{x}-\mb{u}\rangle$,~$\forall\mb{x}\in\mbb{R}^p$, where $\mb{u}$ is the optimal solution of the Problem P1.
	\end{enumerate}
\end{assump}
\noindent Following Assumption~\ref{asp1}, we have for any~$\mb{x},\mb{y}\in\mbb{R}^{np}$,
\begin{subequations}\label{grad}
\begin{align}
\left\|\nabla \mb{f}(\mb{x})-\nabla \mb{f}(\mb{y})\right\|&\leq L_{f}\left\|\mb{x}-\mb{y}\right\|,\label{grada}\\
S_{f}\left\|\mb{x}-\mb{z}^*\right\|^2&\leq\left\langle\nabla \mb{f}(\mb{x})-\nabla \mb{f}(\mb{z}^*),\mb{x}-\mb{z}^*\right\rangle,\label{gradc}
\end{align}
\end{subequations}
where the constants~$L_f=\max_i\{{L_{f_i}}\}$,~$S_f=\min_i\{{S_{f_i}}\}$, and $\nabla\mb{f}(\mb{x})\triangleq[\nabla f_1(\mb{x}_1);\cdots;\nabla f_n(\mb{x}_n)]$ for any $\mb{x}\triangleq[\mb{x}_1;\cdots;\mb{x}_n]$. 

Recall the definition of~$D^k$ in Eq.~\eqref{D}, we formally denote the limit of~$D^k$ by~$D^\infty$, i.e.,
\begin{align}\label{Dinf}
D^\infty=\lim_{k\rightarrow\infty}D^k=\mbox{diag}\left(A^\infty\cdot\mb{1}_n\right)=\mbox{diag}\left(\boldsymbol{\pi}\right),
\end{align}
where~$\boldsymbol{\pi}$ is some vector in the span of the right-eigenvector of~$A$ corresponding to eigenvalue~$1$. The next assumption is related to the weighting matrices,~$A$,~$\widetilde{A}$, and~$D^\infty$.
\begin{assump}[Weighting matrices]\label{asp2}
	The weighting matrices,~$A$ and~$\widetilde{A}$, used in DEXTRA, Eq.~\eqref{alg1_matrix} or~\eqref{alg2}, satisfy the following.
	\begin{enumerate}[label=(\alph*)]
		\item $A$ is a column-stochastic matrix.
		\item $\widetilde{A}$ is a column-stochastic matrix and satisfies~$\widetilde{A}=\theta I_n+(1-\theta)A$, for some~$\theta\in(0,\frac{1}{2}]$.
		\item  $\left(D^\infty\right)^{-1}\widetilde{A}+\widetilde{A}^\top\left(D^\infty\right)^{-1}\succ 0$.
	\end{enumerate}
\end{assump}
\noindent 
One way to guarantee Assumption \ref{asp2}(c) is to design the weighting matrix, $\widetilde{A}$, to be \textit{diagonally-dominant}. For example, each agent~$j$ designs the following weights:
\begin{align}
a_{ij}=\left\{
\begin{array}{rl}
1-\zeta (|\mc{N}_j^{{\scriptsize \mbox{out}}}|-1),&i=j,\\
\zeta,&i\neq j,\quad i\in\mc{N}_j^{{\scriptsize \mbox{out}}},
\end{array}
\right.,
\nonumber
\end{align}
where $\zeta$ is some small positive constant close to zero. This weighting strategy guarantees the Assumption \ref{asp2}(c) as we explain in the following. According to the definition of $D^\infty$ in Eq.~\eqref{Dinf}, all eigenvalues of the matrix, $2(D^\infty)^{-1}=(D^\infty)^{-1}I_n+I_n^\top(D^\infty)^{-1}$, are greater than zero. Since eigenvalues are a continuous functions of the corresponding matrix elements,~\cite{eig1,eig2}, there must exist a small constant~$\overline{\zeta}$ such that for all $\zeta\in(0,\overline{\zeta})$ the weighting matrix, $\widetilde{A}$, designed by the constant weighting strategy with parameter $\zeta$, satisfies that all the eigenvalues of the matrix, $(D^\infty)^{-1}\widetilde{A}+\widetilde{A}^\top(D^\infty)^{-1}$, are greater than zero. In Section~\ref{s6}, we show DEXTRA's performance using this strategy.

Since the weighting matrices, $A$ and, $\widetilde{A}$, are designed to be column-stochastic, they satisfy the following.
\begin{lem}\label{lem2}
	(Nedic \textit{et al}.~\cite{opdirect_Nedic}) For any column-stochastic matrix~$A\in\mbb{R}^{n\times n}$, we have
	\begin{enumerate}[label=(\alph*)]
		\item The limit~$\lim_{k\rightarrow\infty}\left[A^k\right]$ exists and~$\lim_{k\rightarrow\infty}\left[A^k\right]_{ij}=\boldsymbol{\pi}_i$, where~$\boldsymbol{\pi}=\{\boldsymbol{\pi}_i\}$ is some vector in the span of the right-eigenvector of~$A$ corresponding to eigenvalue~$1$.
		\item For all~$i\in\{1,\cdots,n\}$, the entries~$\left[A^k\right]_{ij}$ and~$\boldsymbol{\pi}_i$ satisfy
		\begin{align}
		\left|\left[A^k\right]_{ij}-\boldsymbol{\pi}_i\right|<C\gamma^k,\qquad\forall j,\nonumber
		\end{align}
		where we can have~$C=4$ and~$\gamma=(1-\frac{1}{n^n})$.
	\end{enumerate}
\end{lem}
\noindent As a result, we obtain that for any $k$,
\begin{align}\label{DkDinfty}
\left\|D^k-D^\infty\right\|\leq nC\gamma^k.
\end{align}
Eq.~\eqref{DkDinfty} implies that different agents reach consensus in a linear rate with the constant $\gamma$. Clearly, the convergence rate of DEXTRA will not exceed this consensus rate (because the convergence of DEXTRA means both consensus and optimality are achieved). We will show this fact theoretically later in this section. We now denote some notations to simplify the representation in the rest of the paper. Define the following matrices, 
\begin{align}
\label{Meq}M&=(D^\infty)^{-1}\widetilde{A},\\
N&=(D^\infty)^{-1}(\widetilde{A}-A),\label{n}\\
\label{Qeq}Q&=(D^\infty)^{-1}(I_n+A-2\widetilde{A}),\\
\label{Peq}P&=I_n-A,\\
\label{Leq}L&=\widetilde{A}-A,\\
\label{Req}R&=I_n+A-2\widetilde{A},
\end{align}
and constants,
\begin{align}
d&=\max_k\left\{\|D^{k}\|\right\},\label{dconstant}\\
d^-&=\max_k\left\{\|(D^{k})^{-1}\|\right\},\label{d-constant}\\
d_{\infty}^-&=\|(D^\infty)^{-1}\|.\label{d-infconstant}
\end{align}
We also define some auxiliary variables and sequences. Let $\mb{q}^*\in\mbb{R}^{np}$ be some vector satisfying
\begin{align}\label{q*}
\left[L\otimes I_p\right]\mb{q}^*+\alpha\nabla \mb{f}(\mb{z}^*)=\mb{0}_{np};
\end{align}
and $\mb{q}^k$ be the accumulation of $\mb{x}^r$ over time:
\begin{align}\label{q}
\mb{q}^k=\sum_{r=0}^k\mb{x}^r.
\end{align}
Based on $M$, $N$, $D^k$, $\mb{z}^k$, $\mb{z}^*$, $\mb{q}^k$, and $\mb{q}^*$, we further define
\begin{align}\label{t}
G=\left[
\begin{array}{cc}
M^\top\otimes I_p &\\
  &N\otimes I_p
\end{array}
\right],
\mb{t}^k=\left[
\begin{array}{cc}
\left(D^k\otimes I_p\right)\mb{z}^k \\
\mb{q}^k \\
\end{array}
\right],
\mb{t}^*=\left[
\begin{array}{cc}
\left(D^\infty\otimes I_p\right)\mb{z}^* \\
\mb{q}^*
\end{array}
\right].
\end{align}
It is useful to note that the $G$-matrix norm, $\left\|\mb{a}\right\|_G^2$, of any vector, $\mb{a}\in\mbb{R}^{2np}$, is non-negative, i.e., $\left\|\mb{a}\right\|_G^2\geq0$, $\forall \mb{a}$. This is because $G+G^\top$ is PSD as can be shown with the help of the following lemma.
\begin{lem}\label{lem4}
	(Chung.~\cite{d_laplacian}) Let~$\mc{L}_\mc{G}$ denote the Laplacian matrix of a directed graph,~$\mc{G}$. Let~$U$ be a transition probability matrix associated to a Markov chain described on~$\mc{G}$ and~$\boldsymbol{s}$ be the left-eigenvector of~$U$ corresponding to eigenvalue~$1$. Then,
	\begin{align}\label{lem4_eq}
	\mc{L}_\mc{G}=I_n-\frac{S^{1/2}US^{-1/2}+S^{-1/2}U^\top S^{1/2}}{2},\nonumber
	\end{align}
	where~$S=\mbox{diag}(\boldsymbol{s})$. Additionally, if~$\mc{G}$ is strongly-connected, then the eigenvalues of~$\mc{L}_\mc{G}$ satisfy~$0=\lambda_0<\lambda_1<\cdots<\lambda_n$.
\end{lem}
\noindent Considering the underlying directed graph, $\mc{G}$, and let the weighting matrix $A$, used in DEXTRA, be the corresponding transition probability matrix, we obtain that 
\begin{align}
\mc{L}_\mc{G}=&\frac{(D^\infty)^{1/2}(I_n-A^\top)(D^\infty)^{-1/2}}{2}+\frac{(D^\infty)^{-1/2}(I_n-A)(D^\infty)^{1/2}}{2}.
\end{align}
Therefore, we have the matrix $N$, defined in Eq.~\eqref{n}, satisfy
\begin{align}\label{N}
&N+N^\top=2\theta\left(D^\infty\right)^{-1/2}\mc{L}_\mc{G}\left(D^\infty\right)^{-1/2},
\end{align}
where $\theta$ is the positive constant in Assumption \ref{asp2}(b). Clearly,~$N+N^\top$ is PSD as it is a product of PSD matrices and a non-negative scalar. 
%As a result, it follows from Lemma~\ref{lem4} that for any $\mb{a}\in\mbb{R}^{np}$, $\|\mb{a}\|^2_{N\otimes I_p}\geq 0$. 
Additionally, from Assumption \ref{asp2}(c), note that~$M+M^\top$ is PD and thus for any~$\mb{a}\in\mbb{R}^{np}$, it also follows that $\|\mb{a}\|^2_{M^\top\otimes I_p}\geq 0$. Therefore, we conclude that $G+G^\top$ is PSD and for any~$\mb{a}\in\mbb{R}^{2np}$,
\begin{align}\label{Gnorm}
\|\mb{a}\|^2_{G}\geq 0.
\end{align}
We now state the main result of this paper in Theorem~\ref{main_result}.
\begin{theorem}\label{main_result}
Define
\begin{eqnarray*}
C_1&=&d^-\left(d\left\|(I_n+A)\right\|+d\left\|\widetilde{A}\right\|+2\alpha L_f\right),\\ C_2&=&\frac{\left(\lambda_{\max}\left(NN^\top\right)+\lambda_{\max}\left(N+N^\top\right)\right)}{2\widetilde{\lambda}_{\min}\left(L^\top L\right)},\\
C_3&=&\alpha(nC)^2\left[\frac{C_1^2}{2\eta}+(d^-_\infty d^-L_f)^2\left(\eta+\frac{1}{\eta}\right)+\frac{ S_f}{d^2}\right],\\ 
C_4&=&8C_2\left(L_fd^-\right)^2,\\
C_5&=&\lambda_{\max}\left(\frac{M+M^\top}{2}\right)+4C_2\lambda_{\max}\left(R^\top R \right),\\
C_6&=&\frac{\frac{S_f}{d^2}-\eta-2\eta(d_\infty^-d^-L_f)^2}{2},\\
C_7&=&\frac{1}{2}\lambda_{\max}\left(MM^\top\right)+4C_2\lambda_{\max}\left(\widetilde{A}^\top \widetilde{A} \right),\\
  \Delta&=&C_6^2-4C_4\delta\left(\frac{1}{\delta}+C_5\delta\right),
\end{eqnarray*}
where $\eta$ is some positive constant satisfying that $0<\eta<\frac{S_f}{d^2(1+(d_\infty^-d^-L_f)^2)}$, and $\delta<\lambda_{\min}(M+M^\top)/(2C_7)$ is a positive constant reflecting the convergence rate. 

Let Assumptions~\ref{asp1} and~\ref{asp2} hold.  Then with proper step-size~$\alpha\in\left[\alpha_{\min}, \alpha_{\max}\right]$, there exist,~$0<\Gamma<\infty$ and~$0<\gamma<1$, such that the sequence~$\left\{\mb{t}^k\right\}$ defined in Eq.~\eqref{t} satisfies
\begin{align}\label{main_result_eq}
\left\|\mb{t}^k-\mb{t}^*\right\|_G^2\geq(1+\delta)\left\|\mb{t}^{k+1}-\mb{t}^*\right\|_G^2-\Gamma\gamma^k.
\end{align}
The constant $\gamma$ is the same as used in Eq. \eqref{DkDinfty}, reflecting the consensus rate.
The lower bound,~$\alpha_{\min}$, of~$\alpha$ satisfies $\alpha_{\min}\leq\underline{\alpha}$, where
\begin{align}\label{amin}
\underline{\alpha}\triangleq\frac{C_6-\sqrt{\triangle}}{2C_4\delta},
\end{align}
and the upper bound,~$\alpha_{\max}$, of~$\alpha$ satisfies $\alpha_{\max}\geq\overline{\alpha}$, where
\begin{align}\label{amax}
\overline{\alpha}\triangleq\min\left\{\frac{\eta\lambda_{min}\left(M+M^\top\right)}{2(d_{\infty}^-d^-L_f)^2},\frac{C_6+\sqrt{\triangle}}{2C_4\delta}\right\}.
\end{align}
\end{theorem}
\begin{proof}
	See Section \ref{s5}.
\end{proof}
\noindent Theorem \ref{main_result} is the key result of this paper. We will show the complete proof of Theorem \ref{main_result} in Section \ref{s5}. Note that Theorem 1 shows the relation between $\|\mb{t}^k-\mb{t}^*\|_G^2$ and $\|\mb{t}^{k+1}-\mb{t}^*\|_G^2$ but we would like to show that $\mb{z}^k$ converges linearly to the optimal point $\mb{z}^*$, which Theorem \ref{main_result} does not show. To this aim, we provide Theorem \ref{main_result2} that describes a relation between $\|\mb{z}^k-\mb{z}^*\|^2$ and $\|\mb{z}^{k+1}-\mb{z}^*\|^2$.

In Theorem~\ref{main_result}, we are given specific bounds on $\alpha_{\min}$ and $\alpha_{\max}$. In order to ensure that the solution set of step-size, $\alpha$, is not empty, i.e., $\alpha_{\min}\leq\alpha_{\max}$, it is sufficient (but not necessary) to satisfy
\begin{align}
&\underline{\alpha}=\frac{C_6-\sqrt{\triangle}}{2C_4\delta}\leq\frac{\eta\lambda_{min}\left(M+M^\top\right)}{2(d_{\infty}^-d^-L_f)^2}\leq\overline{\alpha}\label{r2},
\end{align}
which is equivalent to 
\begin{align}\label{eta1}
\eta\geq\frac{\left(\frac{S_f}{2d^2}-\sqrt{\Delta}\right)/\left(2C_4\delta\right)}{\frac{\lambda_{min}\left(M+M^\top\right)}{2L_f^2(d_{\infty}^-d^-)^2}+\frac{1+2(d_\infty^-d^-L_f)^2}{4C_4\delta}}.
\end{align}
Recall from Theorem \ref{main_result} that
\begin{align}\label{eta2}
\eta\leq\frac{S_f}{d^2(1+(d_\infty^-d^-L_f)^2)}.
\end{align}
We note that it may not always be possible to find solutions for $\eta$ that satisfy both Eqs.~\eqref{eta1} and \eqref{eta2}. The theoretical restriction here is due to the fact that the step-size bounds in Theorem~\ref{main_result} are not tight. However, the representation of~$\underline{\alpha}$ and~$\overline{\alpha}$ imply how to increase the interval of appropriate step-sizes. For example, it may be useful to set the weights to increase~$\lambda_{\min}\left(M+N^\top\right)/(2d_{\infty^-}d^-)^2$ such that~$\overline{\alpha}$ is increased. We will discuss such strategies in the numerical experiments in Section~\ref{s6}. We also observe that in reality, the range of appropriate step-sizes is much wider. Note that the values of~$\underline{\alpha}$ and~$\overline{\alpha}$ need the knowledge of the network topology, which may not be available in a distributed manner. Such bounds are not uncommon in the literature where the step-size is a function of the entire topology or global objective functions, see \cite{DGD_Yuan,EXTRA}. It is an open question on how to avoid the global knowledge of network topology when designing the interval of~$\alpha$.
\begin{rem}
	The positive constant $\delta$ in Eq. \eqref{main_result_eq} reflects the convergence rate of $\|\mb{t}^k-\mb{t}^*\|_G^2$. The larger $\delta$ is, the faster $\|\mb{t}^k-\mb{t}^*\|_G^2$ converges to zero. As $\delta<\lambda_{\max}(M+M^\top)/(2C_7)$, we claim that the convergence rate of $\|\mb{t}^k-\mb{t}^*\|_G^2$ can not be arbitrarily large.
\end{rem}
Based on Theorem \ref{main_result}, we now show the $r$-linear convergence rate of DEXTRA to the optimal solution. 
\begin{theorem}\label{main_result2}
Let Assumptions~\ref{asp1} and~\ref{asp2} hold. With the same step-size, $\alpha$, used in Theorem \ref{main_result}, the sequence, $\{\mb{z}^k\}$, generated by DEXTRA, converges exactly to the optimal solution, $\mb{z}^*$, at an $r$-linear rate, i.e., there exist some bounded constants, $T>0$ and $\max\left\{\frac{1}{1+\delta},\gamma\right\}<\tau<1$, where $\delta$ and $\gamma$ are constants used in Theorem \ref{main_result}, Eq.~\eqref{main_result_eq}, such that for any $k$,
\begin{align}
\left\|\left(D^k\otimes I_p\right)\mb{z}^k-\left(D^\infty\otimes I_p\right)\mb{z}^*\right\|^2\leq T\tau^k.\nonumber
\end{align}
\end{theorem}
\begin{proof}
We start with Eq.~\eqref{main_result_eq} in Theorem~\ref{main_result}, which is defined earlier in Eq.~\eqref{t}. Since the $G$-matrix norm is non-negative. recall Eq.~\eqref{Gnorm}, we have $\|\mb{t}^{k}-\mb{t}^*\|_G^2\geq0$, for any $k$. Define $\psi=\max\left\{\frac{1}{1+\delta},\gamma\right\}$, where $\delta$ and $\gamma$ are constants in Theorem \ref{main_result}. From Eq.~\eqref{main_result_eq}, we have for any $k$,
\begin{align}
\left\|\mb{t}^k-\mb{t}^*\right\|_G^2&\leq\frac{1}{1+\delta}\left\|\mb{t}^{k-1}-\mb{t}^*\right\|_G^2+\Gamma\frac{\gamma^{k-1}}{1+\delta},\nonumber\\
&\leq\psi\left\|\mb{t}^{k-1}-\mb{t}^*\right\|_G^2+\Gamma\psi^k,\nonumber\\
&\leq\psi^k\left\|\mb{t}^{0}-\mb{t}^*\right\|_G^2+k\Gamma\psi^k.\nonumber
\end{align}
For any $\tau$ satisfying $\psi<\tau<1$, there exists a constant $\Psi$ such that $(\frac{\tau}{\psi})^k>\frac{k}{\Psi}$, for all $k$. Therefore, we obtain that
\begin{align}
\left\|\mb{t}^k-\mb{t}^*\right\|_G^2&\leq\tau^k\left\|\mb{t}^{0}-\mb{t}^*\right\|_G^2+(\Psi\Gamma)\frac{k}{\Psi}\left(\frac{\psi}{\tau}\right)^k\tau^k,\nonumber\\
&\leq\left(\left\|\mb{t}^{0}-\mb{t}^*\right\|_G^2+\Psi\Gamma\right)\tau^k.\label{t2_2}
\end{align}
From Eq.~\eqref{t} and the corresponding discussion, we have
\begin{align}
\left\|\mb{t}^k-\mb{t}^*\right\|_G^2=&\left\|\left(D^k\otimes I_p\right)\mb{z}^k-\left(D^\infty\otimes I_p\right)\mb{z}^*\right\|^2_{M^\top}+\left\|\mb{q}^k-\mb{q}^*\right\|^2_{N}.\nonumber
\end{align}
Since $N+N^\top$ is PSD, (see Eq.~\eqref{N}), it follows that
\begin{align}
\left\|\left(D^k\otimes I_p\right)\mb{z}^k-\left(D^\infty\otimes I_p\right)\mb{z}^*\right\|^2_{\frac{M+M^\top}{2}}\leq\left\|\mb{t}^k-\mb{t}^*\right\|_G^2.\nonumber
\end{align}
Noting that $M+M^\top$ is PD (see Assumption~\ref{asp2}(c)), i.e., all eigenvalues of $M+M^\top$ are positive, we obtain that
\begin{align}
&\left\|\left(D^k\otimes I_p\right)\mb{z}^k-\left(D^\infty\otimes I_p\right)\mb{z}^*\right\|^2_{\frac{\lambda_{\min}(M+M^\top)}{2}I_{np}}\leq\left\|\left(D^k\otimes I_p\right)\mb{z}^k-\left(D^\infty\otimes I_p\right)\mb{z}^*\right\|^2_{\frac{M+M^\top}{2}}.\nonumber
\end{align}
Therefore, we have that
\begin{align}
\frac{\lambda_{\min}(M+M^\top)}{2}\left\|\left(D^k\otimes I_p\right)\mb{z}^k-\left(D^\infty\otimes I_p\right)\mb{z}^*\right\|^2&\leq\left\|\mb{t}^k-\mb{t}^*\right\|_G^2.\nonumber\\
&\leq\left(\left\|\mb{t}^{0}-\mb{t}^*\right\|_G^2+\Psi\Gamma\right)\tau^k\nonumber.
\end{align}
By letting
\begin{align}
T=2\frac{\|\mb{t}^{0}-\mb{t}^*\|_G^2+\Psi\Gamma}{\lambda_{\min}(M+M^\top)},\nonumber
\end{align}
we obtain the desired result.
\end{proof}
\noindent Theorem \ref{main_result2} shows that the sequence, $\{\mb{z}^k\}$, converges at an $r$-linear rate to the optimal solution, $\mb{z}^*$, where the convergence rate is described by the constant, $\tau$. During the derivation of~$\tau$, we have~$\tau$ satisfying that $\gamma\leq\max\{\frac{1}{1+\delta},\gamma\}<\tau<1$. This implies that the convergence rate (described by the constant $\tau$) is bounded by the consensus rate (described by the constant $\gamma$). In Sections \ref{s4} and \ref{s5}, we present some basic relations and the proof of Theorem \ref{main_result}.

\section{Auxiliary Relations}~\label{s4}
We provide several basic relations in this section, which will help in the proof of Theorem \ref{main_result}. For the proof, we will assume that the sequences updated by DEXTRA have only one dimension, i.e., ~$p=1$; thus $z_i^k$,~$x_i^k\in\mbb{R},\forall i,k$. For~$\mb{x}_i^k,\mb{z}_i^k\in\mbb{R}^p$ being~$p$-dimensional vectors, the proof is the same for every dimension by applying the results to each coordinate. Therefore, assuming~$p=1$ is without the loss of generality. Let~$p=1$ and rewrite DEXTRA, Eq.~\eqref{alg2}, as
\begin{align}\label{alg3}
D^{k+1}\mb{z}^{k+1}=&(I_n+A)D^{k}\mb{z}^{k}-\widetilde{A}D^{k-1}\mb{z}^{k-1}-\alpha\left[\nabla\mb{f}(\mb{z}^k)-\nabla\mb{f}(\mb{z}^{k-1})\right].
\end{align}
We first establish a relation among~$D^k\mb{z}^k$,~$\mb{q}^k$,~$D^\infty\mb{z}^*$, and~$\mb{q}^*$, recall the notation and discussion after Lemma~\ref{lem2}).
\begin{lem}\label{lem1}
	Let Assumptions~\ref{asp1} and~\ref{asp2} hold. In DEXTRA, the quadruple sequence~$\{D^k\mb{z}^k, \mb{q}^k, D^\infty\mb{z}^*, \mb{q}^*\}$ obeys, for any $k$,
	\begin{align}\label{lem1_eq}
	&R\left(D^{k+1}\mb{z}^{k+1}-D^\infty\mb{z}^*\right)+\widetilde{A}\left(D^{k+1}\mb{z}^{k+1}-D^{k}\mb{z}^{k}\right)=-L\left(\mb{q}^{k+1}-\mb{q}^*\right)-\alpha\left[\nabla\mb{f}(\mb{z}^k)-\nabla\mb{f}(\mb{z}^{*})\right],
	\end{align}
	recall Eqs.~\eqref{Meq}--\eqref{q} for notation. 
\end{lem}
\begin{proof}
	We sum DEXTRA, Eq.~\eqref{alg3}, over time from~$0$ to~$k$,
	\begin{align}\nonumber
	D^{k+1}\mb{z}^{k+1}=\widetilde{A}D^k\mb{z}^k-\alpha\nabla\mb{f}(\mb{z}^k)-L\sum_{r=0}^{k}D^r\mb{z}^r.
	\end{align}
	By subtracting~$LD^{k+1}\mb{z}^{k+1}$ on both sides of the preceding equation and rearranging the terms, it follows that
	\begin{equation}\label{lem1_2}
	RD^{k+1}\mb{z}^{k+1}+\widetilde{A}\left(D^{k+1}\mb{z}^{k+1}-D^{k}\mb{z}^{k}\right)
	=-L\mb{q}^{k+1}-\alpha\nabla \mb{f}\left(\mb{z}^k\right).
	\end{equation}
	Note that $D^\infty\mb{z}^*=\boldsymbol{\pi}$, where~$\boldsymbol{\pi}$ is some vector in the span of the right-eigenvector of~$A$ corresponding to eigenvalue~$1$. Since~$R\boldsymbol{\pi}=\mb{0}_n$, we have
	\begin{align}\label{lem1_3}
	RD^\infty\mb{z}^*=\mb{0}_n.
	\end{align}
	By subtracting Eq.~\eqref{lem1_3} from Eq.~\eqref{lem1_2}, and noting that $L\mb{q}^*+\alpha\nabla \mb{f}(\mb{z}^*)=\mb{0}_{n}$, Eq.~\eqref{q*}, we obtain the desired result.
\end{proof}
Recall Eq.~\eqref{DkDinfty} that shows the convergence of $D^k$ to~$D^\infty$ at a geometric rate. We will use this result to develop a relation between $\left\|D^{k+1}\mb{z}^{k+1}-D^k\mb{z}^k\right\|$ and~$\left\|\mb{z}^{k+1}-\mb{z}^k\right\|$, which is in the following lemma.  Similarly, we can establish a relation between $\left\|D^{k+1}\mb{z}^{k+1}-D^\infty\mb{z}^*\right\|$ and $\left\|\mb{z}^{k+1}-\mb{z}^*\right\|$.
\begin{lem}\label{lem3}
	Let Assumptions~\ref{asp1} and~\ref{asp2} hold and recall the constants $d$ and $d^-$ from Eqs. \eqref{dconstant} and \eqref{d-constant}. If $\mb{z}^k$ is bounded, i.e., $\|\mb{z}^k\|\leq B<\infty$, then 
	%it follows that $\left\|D^{k+1}\mb{z}^{k+1}-D^\infty\mb{z}^*\right\|$,~$\left\|\mb{z}^{k+1}-\mb{z}^*\right\|$,~$\left\|D^{k+1}\mb{z}^{k+1}-D^k\mb{z}^k\right\|$, and~$\left\|\mb{z}^{k+1}-\mb{z}^k\right\|$ satisfy the following:
	\begin{enumerate}[label=(\alph*)]
		\item $
		\left\|\mb{z}^{k+1}-\mb{z}^k\right\|\leq d^-\left\|D^{k+1}\mb{z}^{k+1}-D^k\mb{z}^k\right\|+2d^-nCB\gamma^k;$
		\item $
		\left\|\mb{z}^{k+1}-\mb{z}^*\right\|\leq d^-\left\|D^{k+1}\mb{z}^{k+1}-D^\infty\mb{z}^*\right\|+d^-nCB\gamma^k;$
		\item $
		\left\|D^{k+1}\mb{z}^{k+1}-D^\infty\mb{z}^*\right\|\leq d\left\|\mb{z}^{k+1}-\mb{z}^*\right\|+nCB\gamma^k;
		$%
	\end{enumerate}
	where $C$ and $\gamma$ are constants defined in Lemma \ref{lem2}.
\end{lem}
\begin{proof}
	(a)\begin{align}
	\left\|\mb{z}^{k+1}-\mb{z}^k\right\|&=\left\|\left(D^{k+1}\right)^{-1}\left(D^{k+1}\right)\left(\mb{z}^{k+1}-\mb{z}^k\right)\right\|,\nonumber\\
	&\leq\left\|\left(D^{k+1}\right)^{-1}\right\|\left\|D^{k+1}\mb{z}^{k+1}-D^{k}\mb{z}^{k}+D^{k}\mb{z}^{k}-D^{k+1}\mb{z}^k\right\|,\nonumber\\
	&\leq d^-\left\|D^{k+1}\mb{z}^{k+1}-D^k\mb{z}^k\right\|+d^-\left\|D^k-D^{k+1}\right\|\left\|\mb{z}^k\right\|,\nonumber\\
	&\leq d^-\left\|D^{k+1}\mb{z}^{k+1}-D^k\mb{z}^k\right\|+2d^-nCB\gamma^k.\nonumber
	\end{align}
	Similarly, we can prove (b). Finally, we have
	\begin{align}
	\left\|D^{k+1}\mb{z}^{k+1}-D^\infty\mb{z}^*\right\|&=\left\|D^{k+1}\mb{z}^{k+1}-D^{k+1}\mb{z}^{*}+D^{k+1}\mb{z}^{*}-D^{\infty}\mb{z}^*\right\|\nonumber,\\
	&\leq d\left\|\mb{z}^{k+1}-\mb{z}^*\right\|+\left\|D^{k+1}-D^\infty\right\|\left\|\mb{z}^*\right\|,\nonumber\\
	&\leq d\left\|\mb{z}^{k+1}-\mb{z}^*\right\|+nCB\gamma^k.\nonumber
\end{align}
The proof is complete.
\end{proof}
\noindent Note that the result of Lemma \ref{lem3} is based on the prerequisite that the sequence $\{\mb{z}^k\}$ generated by DEXTRA at $k$th iteration is bounded. We will show this boundedness property (for all $k$) together with the proof of Theorem \ref{main_result} in the next section. The following two lemmas discuss the boundedness of $\|\mb{z}^k\|$ for a fixed $k$.
\begin{lem}\label{lem5}
	Let Assumptions~\ref{asp1} and~\ref{asp2} hold and recall $\mb{t}^k$, $\mb{t}^*$, and $G$ defined in Eq.~\eqref{t}. If $\|\mb{t}^k-\mb{t}^*\|_G^2$ is bounded by some constant $F$ for some $k$, i.e.,$\|\mb{t}^k-\mb{t}^*\|_G^2\leq F$, we have $\|\mb{z}^k\|$ be bounded by a constant $B$ for the same $k$, defined as follow,
	\begin{align}\label{pf_asp}
	\left\|\mb{z}^k\right\|\leq B\triangleq\sqrt{\frac{2(d^-)^2F}{\lambda_{\min}\left(\frac{M+M^\top}{2}\right)}+2(d^-)^2\left\|D^\infty\mb{z}^*\right\|^2},
	\end{align}
	where $d^-$, $M$ are constants defined in Eq.~\eqref{d-constant} and \eqref{Meq}.
\end{lem}
\begin{proof}
	We follow the following derivation,
\begin{align}
\frac{1}{2}\left\|\mb{z}^k\right\|^2&\leq\frac{(d^-)^2}{2}\left\|D^k\mb{z}^k\right\|^2\nonumber,\\
&\leq (d^-)^2\left\|D^k\mb{z}^k-D^\infty\mb{z}^*\right\|^2+(d^-)^2\left\|D^\infty\mb{z}^*\right\|^2\nonumber,\\\nonumber
&\leq\frac{(d^-)^2}{\lambda_{\min}\left(\frac{M+M^\top}{2}\right)}\left\|D^k\mb{z}^k-D^\infty\mb{z}^*\right\|^2_{M^\top}+(d^-)^2\left\|D^\infty\mb{z}^*\right\|^2\nonumber,\\\nonumber
&\leq\frac{(d^-)^2}{\lambda_{\min}\left(\frac{M+M^\top}{2}\right)}\left\|\mb{t}^k-\mb{t}^*\right\|^2_{G}+(d^-)^2\left\|D^\infty\mb{z}^*\right\|^2\nonumber,\\\nonumber
&\leq\frac{(d^-)^2F}{\lambda_{\min}\left(\frac{M+M^\top}{2}\right)}+(d^-)^2\left\|D^\infty\mb{z}^*\right\|^2\nonumber,
\end{align}
where the third inequality holds due to $M+M^\top$ being PD (see Assumption \ref{asp2}(c)), and the fourth inequality holds because $N$-matrix norm has been shown to be nonnegative (see Eq. \eqref{N}). Therefore, it follows that $\left\|\mb{z}^k\right\|\leq B$ for $B$ defined in Eq.~\eqref{pf_asp}, which is clearly $<\infty$ as long as $F<\infty.$
\end{proof}
\begin{lem}\label{lem6}
	Let Assumptions~\ref{asp1} and~\ref{asp2} hold and recall the definition of constant $C_1$ from Theorem \ref{main_result}. If $\|\mb{z}^{k-1}\|$ and $\|\mb{z}^{k}\|$ are bounded by a same constant $B$, we have that $\|\mb{z}^{k+1}\|$ is also bounded. More specifically, we have $\|\mb{z}^{k+1}\|\leq C_1B$.
\end{lem}
\begin{proof}
According to the iteration of DEXTRA in Eq.~\eqref{alg3}, we can bound $D^{k+1}\mb{z}^{k+1}$ as
\begin{align}
\left\|D^{k+1}\mb{z}^{k+1}\right\|\leq&\left\|(I_n+A)D^{k}\right\|\left\|\mb{z}^{k}\right\|+\left\|\widetilde{A}D^{k-1}\right\|\left\|\mb{z}^{k-1}\right\|+\alpha L_f\left\|\mb{z}^k\right\|+\alpha L_f\left\|\mb{z}^{k-1}\right\|,\nonumber\\
\leq&\left[d\left\|(I_n+A)\right\|+d\left\|\widetilde{A}\right\|+2\alpha L_f\right]B,\nonumber
\end{align}
where $d$ is the constant defined in Eq.~\eqref{dconstant}. Accordingly, we have $\mb{z}^{k+1}$ be bounded as follow,
\begin{align}\label{zk+1}
\left\|\mb{z}^{k+1}\right\|\leq d^-\left\|D^{k+1}\mb{z}^{k+1}\right\|=C_1B.
\end{align}
\end{proof}

\section{Main Results}~\label{s5}
In this section, we first give two propositions that provide the main framework of the proof. Based on these propositions, we use induction to prove Theorem \ref{main_result}. Proposition \ref{prop1} claims that for all $k\in\mbb{N}^+$, if $\|\mb{t}^{k-1}-\mb{t}^*\|_G^2\leq F_1$ and $\|\mb{t}^{k}-\mb{t}^*\|_G^2\leq F_1$, for some bounded constant $F_1$, then, $\|\mb{t}^{k}-\mb{t}^*\|_G^2\geq(1+\delta)\|\mb{t}^{k+1}-\mb{t}^*\|_G^2-\Gamma\gamma^k,$ for some appropriate step-size.
\begin{prop}\label{prop1}
	Let Assumptions~\ref{asp1} and~\ref{asp2} hold, and recall the constants $C_1$, $C_2$, $C_3$, $C_4$, $C_5$, $C_6$, $C_7$, $\Delta$, $\delta$, and $\gamma$ from Theorem \ref{main_result}. Assume $\|\mb{t}^{k-1}-\mb{t}^*\|_G^2\leq F_1$ and $\|\mb{t}^{k}-\mb{t}^*\|_G^2\leq F_1$, for a same bounded constant $F_1$. Let the constant $B$ be a function of $F_1$ as defined in Eq.~\eqref{pf_asp} by substituting $F$ with $F_1$, and we define $\Gamma$ as
	\begin{align}\label{Gammashort}
	\Gamma=&C_3B^2.
	\end{align}
	With proper step-size~$\alpha$, Eq. \eqref{main_result_eq} is satisfied at $k$th iteration, i.e.,
	\begin{align}
	\left\|\mb{t}^k-\mb{t}^*\right\|_G^2\geq(1+\delta)\left\|\mb{t}^{k+1}-\mb{t}^*\right\|_G^2-\Gamma\gamma^k,\nonumber
	\end{align}
	where the range of step-size is given in Eqs.~\eqref{amin} and \eqref{amax} in Theorem \ref{main_result}.
\end{prop}
\begin{proof}
See Appendix \ref{appprop1}.
\end{proof}
\noindent Note that Proposition \ref{prop1} is different from Theorem \ref{main_result} in that: (i) it only proves the result, Eq.~\eqref{main_result_eq}, for a certain $k$, not for all $k\in\mbb{N}^+$; and, (ii) it requires the assumption that $\|\mb{t}^{k-1}-\mb{t}^*\|_G^2\leq F_1$, and $\|\mb{t}^{k}-\mb{t}^*\|_G^2\leq F_1$, for some bounded constant~$F_1$. Next, Proposition \ref{prop2} shows that for all $k\geq K$, where $K$ is some specific value defined later, if $\|\mb{t}^{k}-\mb{t}^*\|_G^2\leq F$, and $\|\mb{t}^{k}-\mb{t}^*\|_G^2\geq(1+\delta)\|\mb{t}^{k+1}-\mb{t}^*\|_G^2-\Gamma\gamma^k$, we have that $\|\mb{t}^{k+1}-\mb{t}^*\|_G^2\leq F$.
\begin{prop}\label{prop2}
	Let Assumptions~\ref{asp1} and~\ref{asp2} hold, and recall the constants $C_1$, $C_2$, $C_3$, $C_4$, $C_5$, $C_6$, $C_7$, $\Delta$, $\delta$, and $\gamma$ from Theorem \ref{main_result}. Assume that at $k$th iteration, $\|\mb{t}^{k}-\mb{t}^*\|_G^2\leq F_2$, for some bounded constant $F_2$, and $\|\mb{t}^{k}-\mb{t}^*\|_G^2\geq(1+\delta)\|\mb{t}^{k+1}-\mb{t}^*\|_G^2-\Gamma\gamma^k$. Then we have that
\begin{align}
\left\|\mb{t}^{k+1}-\mb{t}^*\right\|_G^2\leq F_2
\end{align}
is satisfied for all~$k\geq K$, where~$K$ is defined as
\begin{align}\label{bigK}
K= \ceil*{\log_r\left(\frac{\delta\lambda_{\min}\left(\frac{M+M^\top}{2}\right)}{2\alpha (d^-)^2C_3}\right)}.
\end{align}
\begin{proof}
See Appendix \ref{appprop2}.
\end{proof}
\end{prop}

\subsection{Proof of Theorem~\ref{main_result}}
We now formally state the proof of Theorem \ref{main_result}.
\begin{proof}
Define $F=\max_{1\leq k\leq K}\{\|\mb{t}^k-\mb{t}^*\|_G^2\}$, where $K$ is the constant defined in Eq.~\eqref{bigK}. The goal is to show that Eq.~\eqref{main_result_eq} in Theorem~\ref{main_result} is valid for all $k$ with the step-size being in the range defined in Eqs.~\eqref{amin} and $\eqref{amax}$. 

We first prove the result for $k\in[1,...,K]$: Since $\left\|\mb{t}^k-\mb{t}^*\right\|_G^2\leq F$, $\forall k\in[1,...,K]$, we use the result of Proposition \ref{prop1} to have $$\|\mb{t}^{k}-\mb{t}^*\|_G^2\geq(1+\delta)\|\mb{t}^{k+1}-\mb{t}^*\|_G^2-\Gamma\gamma^k,\quad\forall k\in[1,...,K].$$

Next, we use induction to show Eq.~\eqref{main_result_eq} for all $k\geq K$. For~$F$ defined above:

\noindent (i) Base case: when $k=K$, we have the initial relations that
\begin{subequations}
	\begin{align}
	\left\|\mb{t}^{K-1}-\mb{t}^*\right\|_G^2&\leq F,\\
	\left\|\mb{t}^K-\mb{t}^*\right\|_G^2&\leq F,\\ 
	\|\mb{t}^{K}-\mb{t}^*\|_G^2&\geq(1+\delta)\|\mb{t}^{K+1}-\mb{t}^*\|_G^2-\Gamma\gamma^K.
	\end{align}
\end{subequations}

\noindent (ii) We now assume that the induction hypothesis is true at the $k$th iteration, for some $k\geq K$, i.e.,
\begin{subequations}
	\begin{align}
	\left\|\mb{t}^{k-1}-\mb{t}^*\right\|_G^2&\leq F,\\
	\left\|\mb{t}^k-\mb{t}^*\right\|_G^2&\leq F,\label{induction2b}\\ 
	\|\mb{t}^{k}-\mb{t}^*\|_G^2&\geq(1+\delta)\|\mb{t}^{k+1}-\mb{t}^*\|_G^2-\Gamma\gamma^k\label{induction2c},
	\end{align}
\end{subequations}
and show that this set of equations also hold for $k+1$.

\noindent (iii) Given Eqs. \eqref{induction2b} and \eqref{induction2c}, we obtain $\|\mb{t}^{k+1}-\mb{t}^*\|_G^2\leq F$ by applying Proposition \ref{prop2}. Therefore, by combining $\|\mb{t}^{k+1}-\mb{t}^*\|_G^2\leq F$ with \eqref{induction2b}, we obtain that $\|\mb{t}^{k+1}-\mb{t}^*\|_G^2\geq(1+\delta)\|\mb{t}^{k+2}-\mb{t}^*\|_G^2-\Gamma\gamma^{k+1}$ by Proposition \ref{prop1}. To conclude, we obtain that 
\begin{subequations}
	\begin{align}
	\left\|\mb{t}^{k}-\mb{t}^*\right\|_G^2&\leq F,\\
	\left\|\mb{t}^{k+1}-\mb{t}^*\right\|_G^2&\leq F,\\ 
	\|\mb{t}^{k+1}-\mb{t}^*\|_G^2&\geq(1+\delta)\|\mb{t}^{k+2}-\mb{t}^*\|_G^2-\Gamma\gamma^{k+1}.
	\end{align}
\end{subequations}
hold for $k+1$. 

By induction, we conclude that this set of equations holds for all $k$, which completes the proof.
\end{proof}
\section{Numerical Experiments}\label{s6}
This section provides numerical experiments to study the convergence rate of DEXTRA for a least squares problem over a directed graph. The local objective functions in the least squares problems are strongly-convex. We compare the performance of DEXTRA with other algorithms suited to the case of directed graph: GP as defined by~\cite{opdirect_Nedic,opdirect_Tsianous,opdirect_Tsianous2,opdirect_Tsianous3}, and D-DGD as defined by~\cite{D-DGD}. Our second experiment verifies the existence of~$\alpha_{\min}$ and~$\alpha_{\max}$, such that the proper step-size~$\alpha$ is between~$\alpha_{\min}$ and $\alpha_{\max}$. We also consider various network topologies and weighting strategies to see how the eigenvalues of network graphs effect the interval of step-size,~$\alpha$. Convergence is studied in terms of the residual
\begin{align}\nonumber
re=\frac{1}{n}\sum_{i=1}^n\left\|\mb{z}_i^k-\mb{u}\right\|,
\end{align}
where~$\mb{u}$ is the optimal solution. The distributed least squares problem is described as follows.

Each agent owns a private objective function,~$\mb{h}_i=H_i\mb{x}+\mb{n}_i$, where~$\mb{h}_i\in\mbb{R}^{m_i}$ and~$H_i\in\mbb{R}^{m_i\times p}$ are measured data,~$\mb{x}\in\mbb{R}^p$ is unknown, and~$\mb{n}_i\in\mbb{R}^{m_i}$ is random noise. The goal is to estimate~$\mb{x}$, which we formulate as a distributed optimization problem solving
\begin{align}
\mbox{min }f(\mb{x})=\frac{1}{n}\sum_{i=1}^n\left\|H_i\mb{x}-\mb{h}_i\right\|.\nonumber
\end{align}
We consider the network topology as the digraph shown in Fig.~\ref{graph}. We first apply the local degree weighting strategy, i.e., to assign each agent itself and its out-neighbors equal weights according to the agent's own out-degree, i.e.,
\begin{align}\label{weight1}
a_{ij}=\frac{1}{|\mc{N}_j^{{\scriptsize \mbox{out}}}|},\qquad(i,j)\in\mc{E}.
\end{align}
According to this strategy, the corresponding network parameters are shown in Fig.~\ref{graph}. We now estimate the interval of appropriate step-sizes. We choose $L_f=\max_{i}\{2\lambda_{\max}(H_i^\top H_i)\}=0.14$, and $S_f=\min_{i}\{2\lambda_{\min}(H_i^\top H_i)\}=0.1$. We set $\eta=0.04<S_f/d^2$, and $\delta=0.1$. Note that $\eta$ and $\delta$ are estimated values. According to the calculation, we have $C_1=36.6$ and $C_2=5.6$. Therefore, we estimate that $\overline{\alpha}
=\frac{\eta\lambda_{min}\left(M+M^\top\right)}{2L_f^2(d_{\infty}^-d^-)^2}
=0.26$, and $\underline{\alpha}
<\frac{S_f/(2d^2)-\eta/2}{2C_2\delta}
=9.6\times10^{-4}$. We thus pick $\alpha=0.1\in[\underline{\alpha},~\overline{\alpha}]$ for the following experiments.
\begin{figure}[!htb]
	\begin{center}
\subfigure{\includegraphics[width=6.4cm]{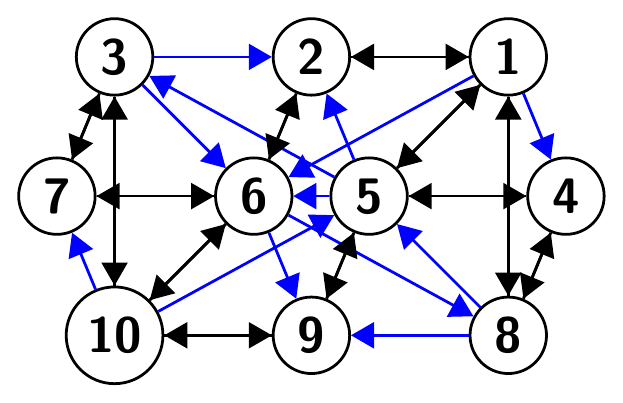}}
\subfigure{\includegraphics[width=6cm]{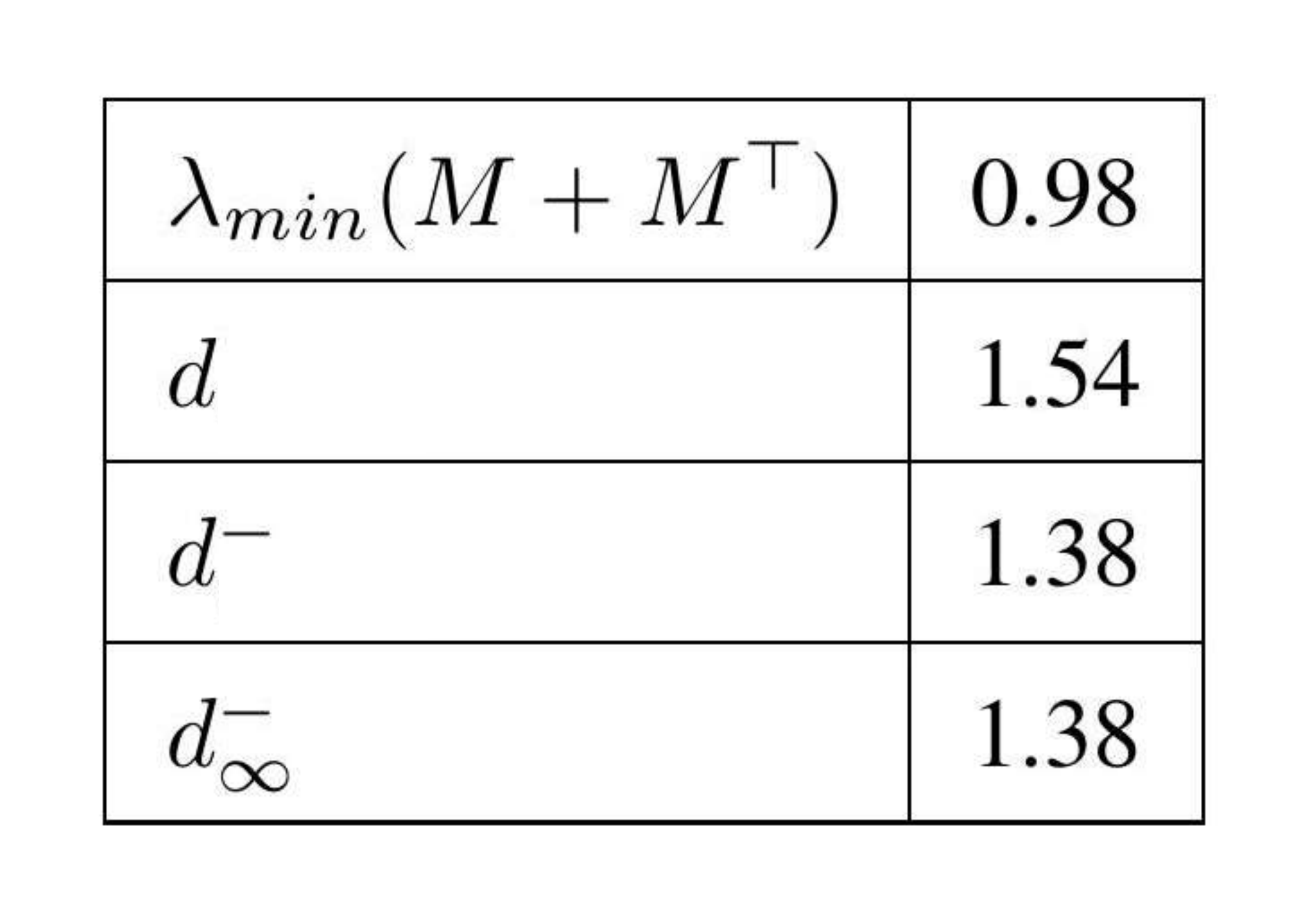}}
\caption{Strongly-connected but non-balanced digraphs and network parameters.}
\label{graph}
\end{center}
\end{figure}

\vspace{0.1cm}
Our first experiment compares several algorithms suited to directed graphs, illustrated in Fig.~\ref{graph}. The comparison of DEXTRA, GP, D-DGD and DGD with weighting matrix being row-stochastic is shown in Fig.~\ref{fig1}. In this experiment, we set~$\alpha=0.1$, which is in the range of our theoretical value calculated above. The convergence rate of DEXTRA is linear as stated in Section~\ref{s3}. G-P and D-DGD apply the same step-size,~$\alpha=\frac{\alpha}{\sqrt{k}}$. As a result, the convergence rate of both is sub-linear. We also consider the DGD algorithm, but with the weighting matrix being row-stochastic. The reason is that in a directed graph, it is impossible to construct a doubly-stochastic matrix. As expected, DGD with row-stochastic matrix does not converge to the exact optimal solution while other three algorithms are suited to directed graphs.
\begin{figure}[!h]
	\begin{center}
		\noindent
		\includegraphics[width=4in]{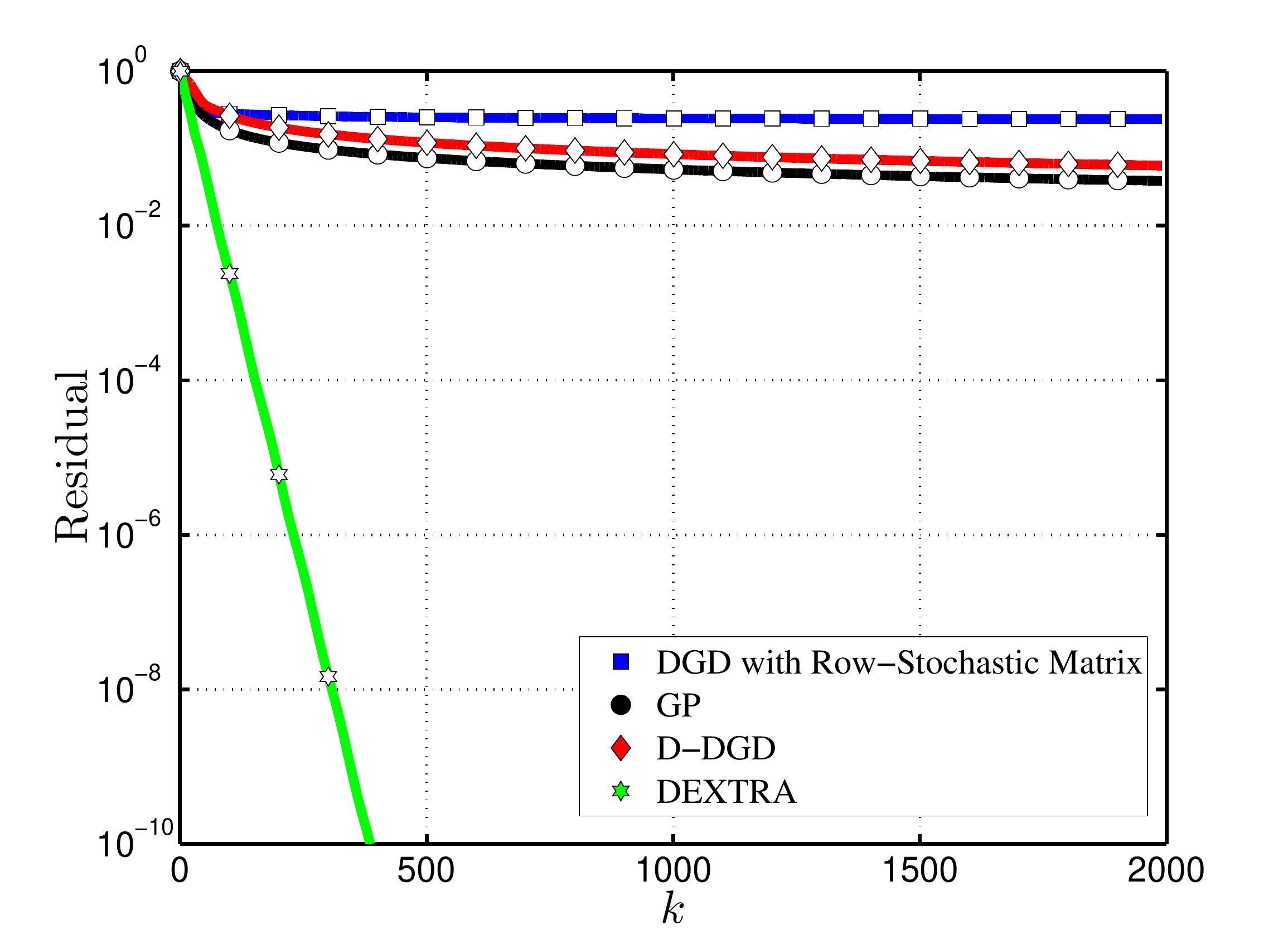}
		\caption{Comparison of different distributed optimization algorithms in a least squares problem. GP, D-DGD, and DEXTRA are proved to work when the network topology is described by digraphs. Moreover, DEXTRA has a linear convergence rate compared with GP and D-DGD.}\label{fig1}
	\end{center}
\end{figure}

According to the theoretical value of~$\underline{\alpha}$ and~$\overline{\alpha}$, we are able to set available step-size,~$\alpha\in[9.6\times10^{-4},0.26]$. In practice, this interval is much wider. Fig.~\ref{fig2} illustrates this fact. Numerical experiments show that~$\alpha_{\min}=0^+$ and~$\alpha_{\max}=0.447$. Though DEXTRA has a much wider range of step-size compared with the theoretical value, it still has a more restricted step-size compared with EXTRA, see~\cite{EXTRA}, where the value of step-size can be as low as any value close to zero in any network topology, i.e.,~$\alpha_{\min}=0$, as long as a \emph{symmetric} and doubly-stochastic matrix is applied in EXTRA. The relative smaller range of interval is due to the fact that the weighting matrix applied in DEXTRA can not be symmetric.
\begin{figure}[!h]
	\begin{center}
		\noindent
		\includegraphics[width=4in]{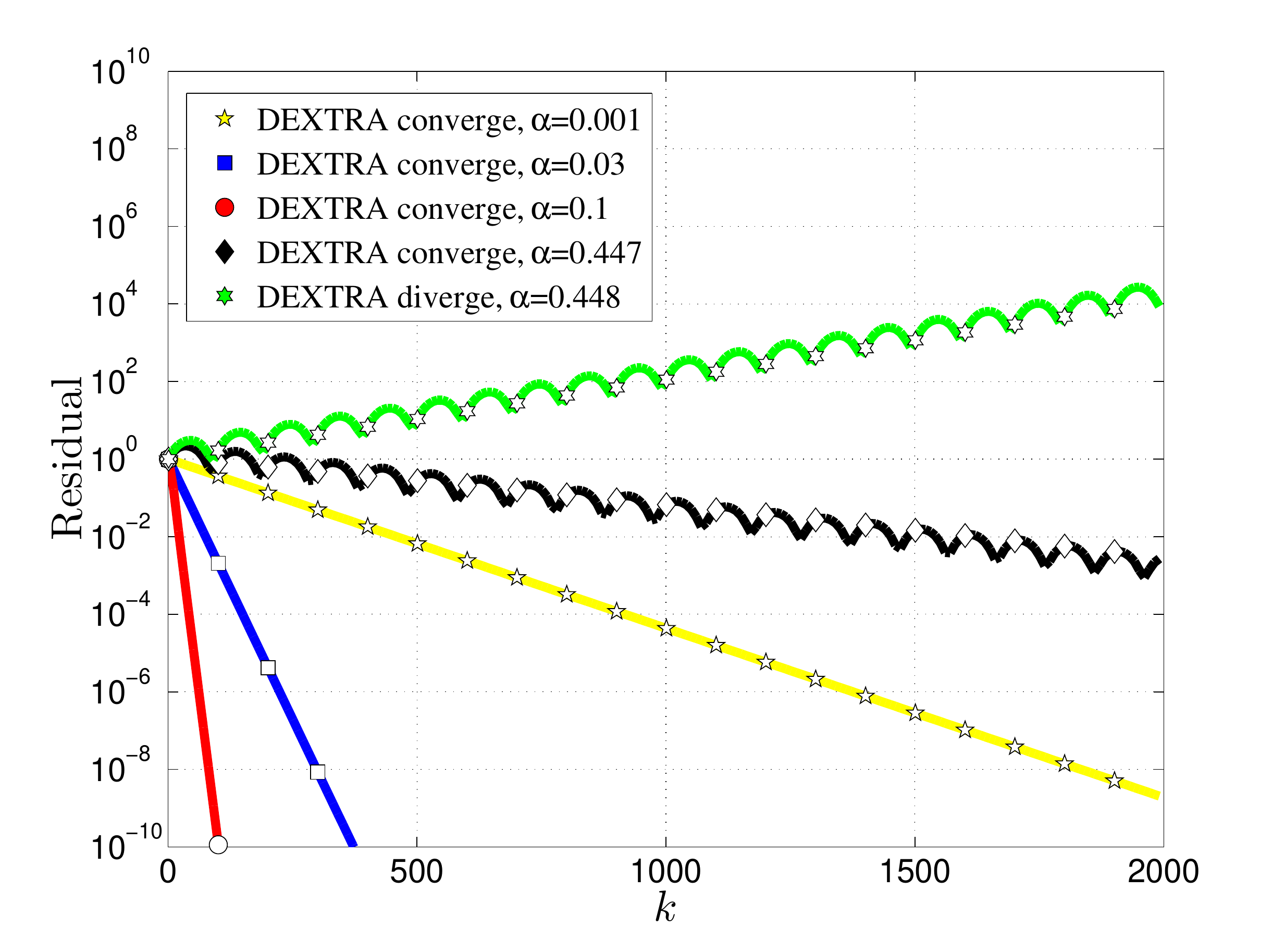}
		\caption{DEXTRA convergence w.r.t. different step-sizes. The practical range of step-size is much wider than theoretical bounds. In this case,~$\alpha\in[\alpha_{min}=0,\alpha_{\max}=0.447] $ while our theoretical bounds show that~$\alpha\in[\underline{\alpha}=5\times10^{-4},\overline{\alpha}=0.26]$.}\label{fig2}
	\end{center}
\end{figure}

The explicit representation of~$\overline{\alpha}$ and~$\underline{\alpha}$ given in Theorem~\ref{main_result} imply the way to increase the interval of step-size, i.e.,
$$\overline{\alpha}\varpropto\frac{\lambda_{min}(M+M^\top)}{(d_{\infty}^-d^-)^2},\qquad\underline{\alpha}\varpropto\frac{1}{(d^-d)^2}.$$ 
To increase $\overline{\alpha}$, we increase~$\frac{\lambda_{min}(M+M^\top)}{(d_{\infty}^-d^-)^2}$; to decrease $\underline{\alpha}$, we can decrease $\frac{1}{(d^-d)^2}$. Compared with applying the local degree weighting strategy, Eq.~\eqref{weight1}, as shown in Fig.~\ref{fig2}, we achieve a wider range of step-sizes by applying the constant weighting strategy, which can be expressed as
\begin{align}\label{weight2}
a_{ij}=\left\{
\begin{array}{rl}
1-0.01 (|\mc{N}_j^{{\scriptsize \mbox{out}}}|-1),&i=j,\\
0.01,&i\neq j,\quad i\in\mc{N}_j^{{\scriptsize \mbox{out}}},
\end{array}
\right.\quad\forall j,
\nonumber
\end{align}
This constant weighting strategy constructs a \textit{diagonal-dominant} weighting matrix, which increases~$\frac{\lambda_{min}(M+M^\top)}{(d_{\infty}^-d^-)^2}$.
\begin{figure}[!h]
	\begin{center}
		\noindent
		\includegraphics[width=4in]{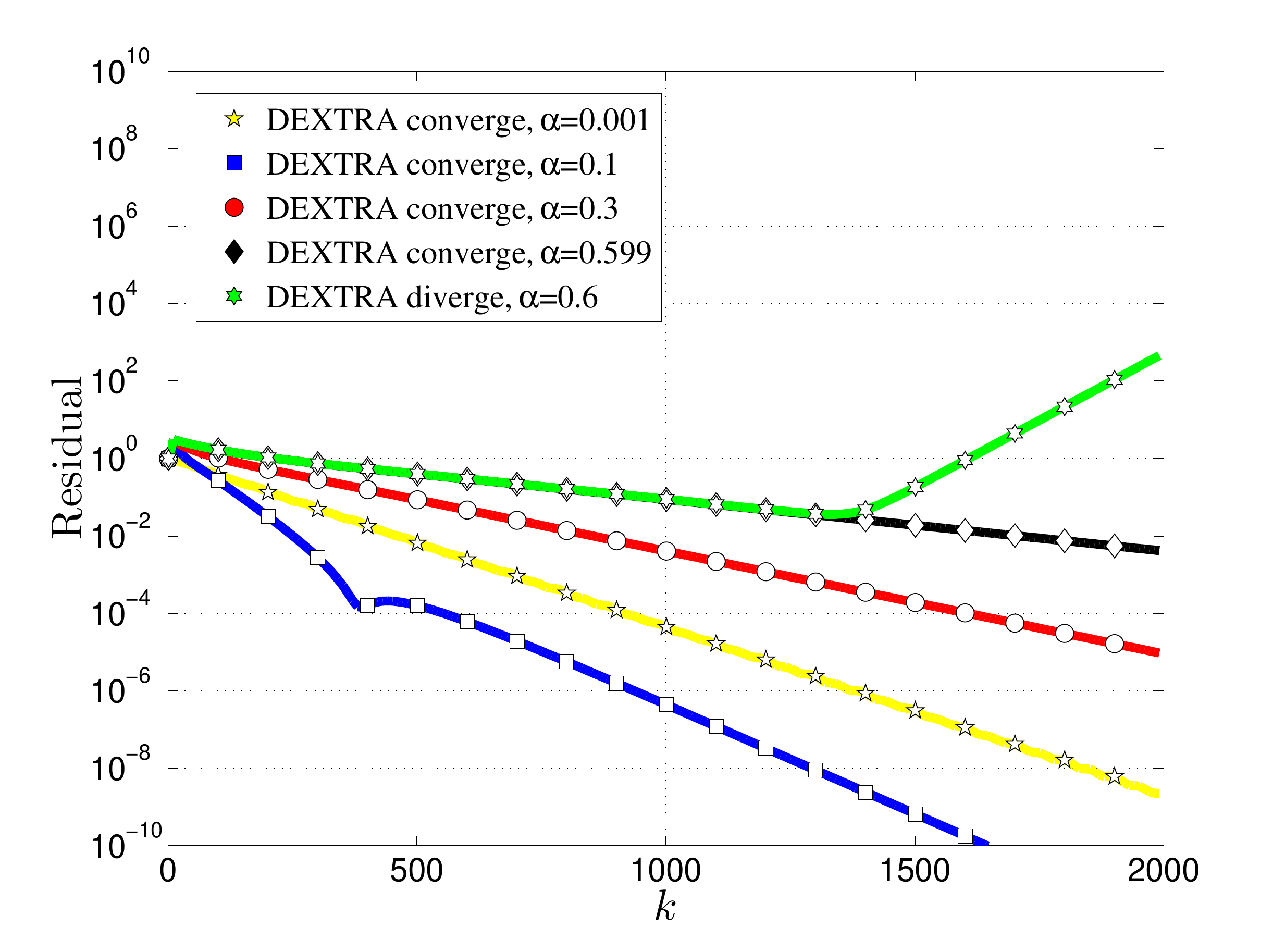}
		\caption{DEXTRA convergence with the weights in Eq.~\eqref{weight2}. A wider range of step-size is obtained due to the increase in $\frac{\lambda_{min}(M+M^\top)}{(d_{\infty}^-d^-)^2}$.}\label{fig2_1}
	\end{center}
\end{figure}
It may also be observed from Figs.~\ref{fig2} and~\ref{fig2_1} that the same step-size generates quiet different convergence speed when the weighting strategy changes. Comparing Figs.~\ref{fig2} and~\ref{fig2_1} when step-size~$\alpha=0.1$, DEXTRA with local degree weighting strategy converges much faster. 

\section{Conclusions}\label{s7}
In this paper, we introduce DEXTRA, a distributed algorithm to solve multi-agent optimization problems over \emph{directed} graphs. We have shown that DEXTRA succeeds in driving all agents to the same point, which is the exact optimal solution of the problem, given that the communication graph is strongly-connected and the objective functions are strongly-convex. Moreover, the algorithm converges at a linear rate~$O(\tau^k)$ for some constant,~$\tau<1$. Numerical experiments on a least squares problem show that DEXTRA is the fastest distributed algorithm among all algorithms applicable to directed graphs.

\appendices
\section{Proof of Proposition \ref{prop1}}\label{appprop1}
We first bound $\|\mb{z}^{k-1}\|$, $\|\mb{z}^{k}\|$, and $\|\mb{z}^{k+1}\|$. According to Lemma \ref{lem5}, we obtain that $\|\mb{z}^{k-1}\|\leq B$ and $\|\mb{z}^{k}\|\leq B$, since $\|\mb{t}^{k-1}-\mb{t}^*\|_G^2\leq F_1$ and $\|\mb{t}^{k}-\mb{t}^*\|_G^2\leq F_1$. By applying Lemma \ref{lem6}, we further obtain that $\|\mb{z}^{k+1}\|\leq C_1B$. Based on the boundedness of $\|\mb{z}^{k-1}\|$, $\|\mb{z}^{k}\|$, and $\|\mb{z}^{k+1}\|$, we start to prove the desired result.
By applying the restricted strong-convexity assumption, Eq.~\eqref{gradc}, it follows that 
	\begin{align}%\label{p_1}
	2\alpha S_f\left\|\mb{z}^{k+1}-\mb{z}^*\right\|^2\leq&2\alpha\left\langle D^\infty\left(\mb{z}^{k+1}-\mb{z}^*\right),\left(D^\infty\right)^{-1}\left[\nabla\mb{f}(\mb{z}^{k+1})-\nabla\mb{f}(\mb{z}^*)\right]\right\rangle,\nonumber\\
	=&2\alpha\left\langle D^\infty\mb{z}^{k+1}-D^{k+1}\mb{z}^{k+1},\left(D^\infty\right)^{-1}\left[\nabla\mb{f}(\mb{z}^{k+1})-\nabla\mb{f}(\mb{z}^*)\right]\right\rangle\nonumber\\
	&+2\alpha\left\langle D^{k+1}\mb{z}^{k+1}-D^\infty\mb{z}^*,\left(D^\infty\right)^{-1}\left[\nabla\mb{f}(\mb{z}^{k+1})-\nabla\mb{f}(\mb{z}^k)\right]\right\rangle\nonumber\\
	&+2\left\langle D^{k+1}\mb{z}^{k+1}-D^\infty\mb{z}^*,\left(D^\infty\right)^{-1}\alpha\left[\nabla\mb{f}(\mb{z}^{k})-\nabla\mb{f}(\mb{z}^*)\right]\right\rangle,\nonumber\\
	:=&s_1+s_2+s_3,
	\end{align}
where~$s_1$,~$s_2$,~$s_3$ denote each of RHS terms. We show the boundedness of $s_1$, $s_2$, and $s_3$ in Appendix~\ref{app1}. Next, it follows from Lemma~\ref{lem3}(c) that
\begin{align}
\left\|D^{k+1}\mb{z}^{k+1}-D^\infty\mb{z}^*\right\|^2\leq&2d^2\left\|\mb{z}^{k+1}-\mb{z}^*\right\|^2+2(nCB)^2\gamma^{2k}.\nonumber
\end{align}
Multiplying both sides of the preceding relation by~$\frac{\alpha S_f}{d^2}$ and combining it with Eq. (55), we obtain
\begin{align}\label{lem2_eq_2}
\frac{\alpha S_f}{d^2}\left\|D^{k+1}\mb{z}^{k+1}-D^\infty\mb{z}^*\right\|^2&\leq s_1+s_2+s_3\nonumber\\
&+\frac{2\alpha S_f(nCB)^2}{d^2}\gamma^{2k}.
\end{align}
By plugging the related bounds from Appendix~\ref{app1} ($s_1$ from Eq.~\eqref{ss1}, $s_2$ from Eq.~\eqref{ss2}, and $s_3$ from Eq.~\eqref{ss3_6}) in Eq.~\eqref{lem2_eq_2}, it follows that
\begin{align}\label{lem3_eq_3}
\left\|\mb{t}^k-\mb{t}^*\right\|_G^2-\left\|\mb{t}^{k+1}-\mb{t}^*\right\|_G^2
\geq&\left\|D^{k+1}\mb{z}^{k+1}-D^\infty\mb{z}^*\right\|^2_{\frac{\alpha}{2}\left[\frac{S_f}{d^2}-\eta-2\eta(d_\infty^-d^-L_f)^2\right]I_n-\frac{1}{\delta}I_n+Q}\nonumber\\
&+\left\|D^{k+1}\mb{z}^{k+1}-D^k\mb{z}^k\right\|^2_{M^\top-\frac{\delta}{2}MM^\top-\frac{\alpha (d_{\infty}^-d^-L_f)^2}{\eta}I_n}\nonumber\\
&-\alpha(nC)^2\left[\frac{C_1^2}{2\eta}+(d^-_\infty d^-L_f)^2\left(\eta+\frac{1}{\eta}\right)+\frac{ S_f}{d^2}\right]B^2\gamma^{k}\nonumber\\
&-\left\|\mb{q}^*-\mb{q}^{k+1}\right\|^2_{\frac{\delta}{2}NN^\top}.
\end{align}
In order to derive the relation that
\begin{align}
&\left\|\mb{t}^k-\mb{t}^*\right\|_G^2\geq(1+\delta)\left\|\mb{t}^{k+1}-\mb{t}^*\right\|_G^2-\Gamma\gamma^{k},
\end{align}
it is sufficient to show that the RHS of Eq.~\eqref{lem3_eq_3} is no less than $\delta\left\|\mb{t}^{k+1}-\mb{t}^*\right\|_G^2-\Gamma\gamma^{k}$. Recall the definition of $G$, $\mb{t}^k$, and $\mb{t}^*$ in Eq.~\eqref{t}, we have
\begin{align}\label{lem3_eq_3_1}
\delta\left\|\mb{t}^{k+1}-\mb{t}^*\right\|_G^2-\Gamma\gamma^{k}=&\left\|D^{k+1}\mb{z}^{k+1}-D^\infty\mb{z}^*\right\|^2_{\delta M^\top}+\left\|\mb{q}^*-\mb{q}^{k+1}\right\|^2_{\delta N}-\Gamma\gamma^k.
\end{align}
Comparing Eqs.~\eqref{lem3_eq_3} with \eqref{lem3_eq_3_1}, it is sufficient to prove that
\begin{align}\label{lem3_eq_4}
&\left\|D^{k+1}\mb{z}^{k+1}-D^\infty\mb{z}^*\right\|^2_{\frac{\alpha}{2}\left[\frac{S_f}{d^2}-\eta-2\eta(d_\infty^-d^-L_f)^2\right]I_n-\frac{1}{\delta}I_n+Q-\delta M^\top}\nonumber\\
&+\left\|D^{k+1}\mb{z}^{k+1}-D^k\mb{z}^k\right\|^2_{M^\top-\frac{\delta}{2}MM^\top-\frac{\alpha (d_{\infty}^-d^-L_f)^2}{\eta}I_n}\nonumber\\
& +\Gamma\gamma^k-\alpha(nC)^2\left[\frac{C_1^2}{2\eta}+(d^-_\infty d^-L_f)^2\left(\eta+\frac{1}{\eta}\right)+\frac{ S_f}{d^2}\right]B^2\gamma^{k}\nonumber\\
&\geq\left\|\mb{q}^*-\mb{q}^{k+1}\right\|^2_{\delta\left(\frac{NN^\top}{2}+N\right)}.
\end{align}
We next aim to bound $\|\mb{q}^*-\mb{q}^{k+1}\|^2_{\delta(\frac{NN^\top}{2}+N)}$ in terms of $\|D^{k+1}\mb{z}^{k+1}-D^\infty\mb{z}^*\|$ and $\|D^{k+1}\mb{z}^{k+1}-D^k\mb{z}^k\|$, such that it is easier to analyze Eq.~\eqref{lem3_eq_4}. From Lemma~\ref{lem1}, we have
\begin{align}\label{lem3_eq_6}
&\left\|\mb{q}^*-\mb{q}^{k+1}\right\|^2_{L^\top L}=\left\|L\left(\mb{q}^*-\mb{q}^{k+1}\right)\right\|^2,\nonumber\\
=&\Big{\|}R(D^{k+1}\mb{z}^{k+1}-D^\infty\mb{z}^*)+\alpha[\nabla\mb{f}(\mb{z}^{k+1})-\nabla\mb{f}(\mb{z}^{*})]\nonumber\\
&+\widetilde{A}(D^{k+1}\mb{z}^{k+1}-D^{k}\mb{z}^{k})+\alpha[\nabla\mb{f}(\mb{z}^k)-\nabla\mb{f}(\mb{z}^{k+1})]\Big{\|}^2,\nonumber\\
\leq&4\left(\left\|D^{k+1}\mb{z}^{k+1}-D^\infty\mb{z}^*\right\|^2_{R^\top R}+\alpha^2L_f^2\left\|\mb{z}^{k+1}-\mb{z}^*\right\|^2\right)\nonumber\\
&+4\left(\left\|D^{k+1}\mb{z}^{k+1}-D^k\mb{z}^k\right\|^2_{\widetilde{A}^\top\widetilde{A}}+\alpha^2L_f^2\left\|\mb{z}^{k+1}-\mb{z}^k\right\|^2\right),\nonumber\\
\leq&\left\|D^{k+1}\mb{z}^{k+1}-D^\infty\mb{z}^*\right\|^2_{4R^\top R+8(\alpha L_fd^-)^2I_n}+\left\|D^{k+1}\mb{z}^{k+1}-D^k\mb{z}^k\right\|^2_{4\widetilde{A}^\top\widetilde{A}+8(\alpha L_fd^-)^2I_n}\nonumber\\
&+24\left(\alpha nCd^-L_f\right)^2B^2\gamma^k.
\end{align}
Since that~$\lambda\left(\frac{N+N^\top}{2}\right)\geq0$, $\lambda\left(NN^\top\right)\geq0$, $\lambda\left(L^\top L\right)\geq0$, and $\lambda_{\min}\left(\frac{N+N^\top}{2}\right)=\lambda_{\min}\left(NN^\top\right)=\lambda_{\min}\left(L^\top L\right)=0$ with the same corresponding eigenvector, we have
\begin{align}\label{lem3_eq_7}
\left\|\mb{q}^*-\mb{q}^{k+1}\right\|^2_{\delta\left(\frac{NN^\top}{2}+N\right)}\leq\delta C_2\left\|\mb{q}^*-\mb{q}^{k+1}\right\|^2_{L^\top L},
\end{align}%
where $C_2$ is the constant defined in Theorem \ref{main_result}. By combining Eqs.~\eqref{lem3_eq_6} with \eqref{lem3_eq_7}, it follows that
\begin{align}\label{lem3_eq_7_1}
\left\|\mb{q}^*-\mb{q}^{k+1}\right\|^2_{\delta\left(\frac{NN^\top}{2}+N\right)}&\leq\delta C_2\left\|\mb{q}^*-\mb{q}^{k+1}\right\|^2_{L^\top L},\nonumber\\
&\leq\left\|D^{k+1}\mb{z}^{k+1}-D^\infty\mb{z}^*\right\|^2_{\delta C_2\left(4R^\top R+8(\alpha L_fd^-)^2I_n\right)}\nonumber\\
&+\left\|D^{k+1}\mb{z}^{k+1}-D^k\mb{z}^k\right\|^2_{\delta C_2\left(4\widetilde{A}^\top\widetilde{A}+8(\alpha L_fd^-)^2I_n\right)}\nonumber\\
&+24\delta C_2\left(\alpha nCd^-L_f\right)^2B^2\gamma^k.
\end{align}
Consider Eq.~\eqref{lem3_eq_4}, together with \eqref{lem3_eq_7_1}. Let
\begin{align}%\label{Gammashort}
\Gamma=&C_3B^2,
\end{align}
where $C_3$ is the constant defined in Theorem \ref{main_result}, such that all ``$\gamma^k$ items'' in Eqs.~\eqref{lem3_eq_4} and \eqref{lem3_eq_7_1} can be canceled out. In order to prove Eq.~\eqref{lem3_eq_4}, it is sufficient to show that the LHS of Eq.~\eqref{lem3_eq_4} is no less than the RHS of Eq.~\eqref{lem3_eq_7_1}, i.e.,
\begin{align}\label{lem3_eq_8}
&\left\|D^{k+1}\mb{z}^{k+1}-D^\infty\mb{z}^*\right\|^2_{\frac{\alpha}{2}\left[\frac{S_f}{d^2}-\eta-2\eta(d_\infty^-d^-L_f)^2\right]I_n-\frac{1}{\delta}I_n+Q-\delta M^\top}\nonumber\\
&+\left\|D^{k+1}\mb{z}^{k+1}-D^k\mb{z}^k\right\|^2_{M^\top-\frac{\delta}{2}MM^\top-\frac{\alpha (d_{\infty}^-d^-L_f)^2}{\eta}I_n}\nonumber\\
&\geq\left\|D^{k+1}\mb{z}^{k+1}-D^\infty\mb{z}^*\right\|^2_{\delta C_2\left(4R^\top R+8(\alpha L_fd^-)^2I_n\right)}\nonumber\\
&+\left\|D^{k+1}\mb{z}^{k+1}-D^ks\mb{z}^k\right\|^2_{\delta C_2\left(4\widetilde{A}^\top\widetilde{A}+8(\alpha L_fd^-)^2I_n\right)}.
\end{align}
To satisfy Eq.~\eqref{lem3_eq_8}, it is sufficient to have the following two relations hold simultaneously,
\begin{subequations}
	\begin{align}
	%\left\{\begin{array}{rl}
	&\frac{\alpha}{2}\left[\frac{S_f}{d^2}-\eta-2\eta(d_\infty^-d^-L_f)^2\right]-\frac{1}{\delta}-\delta\lambda_{\max}\left(\frac{M+M^\top}{2}\right)
	\geq\delta C_2\left[4\lambda_{\max}\left(R^\top R\right)+8(\alpha L_fd^-)^2\right],\label{c1}\\
	&\lambda_{\min}\left(\frac{M+M^\top}{2}\right)-\frac{\delta}{2}\lambda_{\max}\left(MM^\top\right)-\frac{\alpha (d_{\infty}^-d^-L_f)^2}{\eta}
	\geq\delta C_2\left[4\lambda_{\max}\left(\widetilde{A}^\top\widetilde{A}\right)+8(\alpha L_fd^-)^2\right].\label{c2}
	%\end{array}\right,
	\end{align}
\end{subequations}
where in Eq.~\eqref{c1} we ignore the term $\frac{\lambda_{\min}\left(Q+Q^\top\right)}{2}$ due to $\lambda_{\min}\left(Q+Q^\top\right)=0$. Recall the definition
\begin{align}
C_4&=8C_2\left(L_fd^-\right)^2,\\
C_5&=\lambda_{\max}\left(\frac{M+M^\top}{2}\right)+4C_2\lambda_{\max}\left(R^\top R \right),\\
C_6&=\frac{\frac{S_f}{d^2}-\eta-2\eta(d_\infty^-d^-L_f)^2}{2},\\
\Delta&=C_6^2-4C_4\delta\left(\frac{1}{\delta}+C_5\delta\right).
\end{align}
The solution of step-size, $\alpha$, satisfying Eq.~\eqref{c1}, is
\begin{align}\label{c4}
\frac{C_6-\sqrt{\Delta}}{2C_4\delta}\leq\alpha\leq\frac{C_6+\sqrt{\Delta}}{2C_4\delta},
\end{align}
where we set
\begin{align}
\eta&<\frac{S_f}{d^2(1+(d_\infty^-d^-L_f)^2)},
\end{align}
to ensure the solution of $\alpha$ contains positive values. In order to have $\delta>0$ in Eq.~\eqref{c2}, the step-size,~$\alpha$, is sufficient to satisfy
\begin{align}\label{c3}
\alpha\leq\frac{\eta\lambda_{min}\left(M+M^\top\right)}{2(d_{\infty}^-d^-L_f)^2}.
\end{align}
By combining Eqs.~\eqref{c4} with \eqref{c3}, we conclude it is sufficient to set the step-size $\alpha\in[\underline{\alpha},\overline{\alpha}]$,
where
\begin{align}
\underline{\alpha}\triangleq\frac{C_6-\sqrt{\triangle}}{2C_4\delta},
\end{align}
and
\begin{align}
\overline{\alpha}\triangleq\min\left\{\frac{\eta\lambda_{min}\left(M+M^\top\right)}{2(d_{\infty}^-d^-L_f)^2},\frac{C_6+\sqrt{\triangle}}{2C_4\delta}\right\},
\end{align}
to establish the desired result, i.e.,
\begin{align}\label{con}
\|\mb{t}^k-\mb{t}^*\|_G^2\geq(1+\delta)\|\mb{t}^{k+1}-\mb{t}^*\|_G^2-\Gamma\gamma^k.
\end{align}
Finally, we bound the constant $\delta$, which reflecting how fast $\left\|\mb{t}^{k+1}-\mb{t}^*\right\|^2_G$ converges. Recall the definition of $C_7$
\begin{align}
C_7=\frac{1}{2}\lambda_{\max}\left(MM^\top\right)+4C_2\lambda_{\max}\left(\widetilde{A}^\top \widetilde{A} \right).
\end{align}
To have $\alpha$'s solution  of Eq.~\eqref{c2} contains positive values, we need to set
\begin{align}
\delta<\frac{\lambda_{\min}\left(M+M^\top\right)}{2C_7}.
\end{align}

\section{Proof of Proposition \ref{prop2}}\label{appprop2}
Since we have $\|\mb{t}^{k}-\mb{t}^*\|_G^2\leq F_2$, and $\|\mb{t}^k-\mb{t}^*\|_G^2\geq(1+\delta)\|\mb{t}^{k+1}-\mb{t}^*\|_G^2-\Gamma\gamma^k$, it follows that
\begin{align}\label{con1}
\left\|\mb{t}^{k+1}-\mb{t}^*\right\|^2_G&\leq\frac{\left\|\mb{t}^{k}-\mb{t}^*\right\|^2_G}{1+\delta}+\frac{\Gamma\gamma^k}{1+\delta},\nonumber\\
&\leq\frac{F_2}{1+\delta}+\frac{\Gamma\gamma^k}{1+\delta}.
\end{align}
Given the definition of $K$ in Eq.~\eqref{bigK}, it follows that for $k\geq K$
\begin{align}
\gamma^k\leq\frac{\delta\lambda_{\min}\left(\frac{M+M^\top}{2}\right)B^2}{2\alpha(d^-)^2C_3B^2}\leq\frac{\delta F_2}{\Gamma},
\end{align}
where the second inequality follows with the definition of $\Gamma$, and $F$ in Eqs.~\eqref{Gammashort} and \eqref{pf_asp}.
Therefore, we obtain that
\begin{align}
\left\|\mb{t}^{k+1}-\mb{t}^*\right\|^2_G\leq\frac{F_2}{1+\delta}+\frac{\delta F_2}{1+\delta}=F_2.
\end{align}
\section{Bounding $s_1,s_2$ and $s_3$}\label{app1}
{\it Bounding $s_1$:} By using $2\langle\mb{a},\mb{b}\rangle\leq\eta\|\mb{a}\|^2+\frac{1}{\eta}\|\mb{b}\|^2$ for any appropriate vectors $\mb{a,b}$, and a positive $\eta$, we obtain that 
\begin{align}
s_1\leq\frac{\alpha}{\eta}\left\|D^\infty-D^{K+1}\right\|^2\left\|\mb{z}^{K+1}\right\|^2+\alpha\eta(d^-_\infty L_f)^2\left\|\mb{z}^{K+1}-\mb{z}^*\right\|^2.
\end{align}
\noindent It follows $\left\|D^\infty-D^{K+1}\right\|\leq nC\gamma^K$ as shown in Eq.~\eqref{DkDinfty}, and $\|\mb{z}^{K+1}\|^2\leq C_1^2B^2$ as shown in Eq.~\eqref{zk+1}. The term $\|\mb{z}^{K+1}-\mb{z}^*\|$ can be bounded with applying Lemma \ref{lem3}(b). Therefore,
\begin{align}\label{ss1}
s_1\leq&\alpha(nC)^2\left[\frac{C_1^2}{\eta}+2\eta(d^-_\infty d^-L_f)^2\right]B^2\gamma^{2K}+2\alpha\eta(d^-_\infty d^-L_f)^2\left\|D^{K+1}\mb{z}^{K+1}-D^\infty\mb{z}^*\right\|^2.
\end{align}
{\it Bounding $s_2$:} Similarly, we use Lemma~\ref{lem3}(a) to obtain 
\begin{align}
s_2\leq&\alpha\eta\left\|D^{K+1}\mb{z}^{K+1}-D^\infty\mb{z}^*\right\|^2+\frac{\alpha(d_\infty^- L_f)^2}{\eta}\left\|\mb{z}^{K+1}-\mb{z}^k\right\|^2\nonumber,\\
\leq&\alpha\eta\left\|D^{K+1}\mb{z}^{K+1}-D^\infty\mb{z}^*\right\|^2+\frac{2\alpha (nCd_{\infty}^-d^-L_f)^2B^2}{\eta}\gamma^{2K}\nonumber\\
&+\frac{2\alpha(d_{\infty}^-d^-L_f)^2}{\eta}\left\|D^{k+1}\mb{z}^{k+1}-D^k\mb{z}^k\right\|^2.\label{ss2}
\end{align}

{\it Bounding $s_3$:} We rearrange Eq.~\eqref{lem1_eq} in Lemma~\ref{lem1} as follow,
\begin{align}
\alpha\left[\nabla\mb{f}(\mb{z}^k)-\nabla\mb{f}(\mb{z}^{*})\right]=R\left(D^{k+1}\mb{z}^{k+1}-D^\infty\mb{z}^*\right)+\widetilde{A}\left(D^{k+1}\mb{z}^{k+1}-D^{k}\mb{z}^{k}\right)+L\left(\mb{q}^{k+1}-\mb{q}^*\right).
\end{align}
By substituting $\alpha[\nabla\mb{f}(\mb{z}^k)-\nabla\mb{f}(\mb{z}^{*})]$ in $s_3$ with the representation in the preceding relation, we represent $s_3$ as
\begin{align}\label{ss3}
s_3=&\left\|D^{K+1}\mb{z}^{K+1}-D^\infty\mb{z}^*\right\|^2_{-2Q}+2\left\langle D^{K+1}\mb{z}^{K+1}-D^\infty\mb{z}^*,M\left(D^{K}\mb{z}^{K}-D^{K+1}\mb{z}^{K+1}\right)\right\rangle\nonumber\\
&+2\left\langle D^{K+1}\mb{z}^{K+1}-D^\infty\mb{z}^*,N\left(\mb{q}^*-\mb{q}^{k+1}\right)\right\rangle,\nonumber\\
:=&s_{3a}+s_{3b}+s_{3c},
\end{align}
where $s_{3b}$ is equivalent to
\begin{align}
s_{3b}=2\left\langle D^{K+1}\mb{z}^{K+1}-D^{K}\mb{z}^{K},M^\top\left(D^\infty\mb{z}^*-D^{K+1}\mb{z}^{K+1}\right)\right\rangle,\nonumber
\end{align}
and $s_{3c}$ can be simplified as
\begin{align}
s_{3c}&=2\left\langle D^{K+1}\mb{z}^{K+1},N\left(\mb{q}^*-\mb{q}^{K+1}\right)\right\rangle\nonumber\\
&=2\left\langle \mb{q}^{K+1}-\mb{q}^K,N\left(\mb{q}^*-\mb{q}^{K+1}\right)\right\rangle.\nonumber
\end{align}
The first equality in the preceding relation holds due to the fact that $N^\top D^{\infty}\mb{z}^*=\mb{0}_n$ and the second equality follows from the definition of $\mb{q}^k$, see Eq.~\eqref{q}. By substituting the representation of $s_{3b}$ and $s_{3c}$ into~\eqref{ss3}, and recalling the definition of~$\mb{t}^k$,~$\mb{t}^*$,~$G$ in Eq.~\eqref{t}, we simplify the representation of $s_3$,
\begin{align}\label{ss3_2}
s_3=&\left\|D^{K+1}\mb{z}^{K+1}-D^\infty\mb{z}^*\right\|^2_{-2Q}+2\left\langle \mb{t}^{K+1}-\mb{t}^K,G\left(\mb{t}^{*}-\mb{t}^{K+1}\right)\right\rangle.
\end{align}
With the basic rule
	\begin{align}
	&\left\langle \mb{t}^{K+1}-\mb{t}^K,G\left(\mb{t}^{*}-\mb{t}^{K+1}\right)\right\rangle+\left\langle G\left(\mb{t}^{K+1}-\mb{t}^K\right),\mb{t}^{*}-\mb{t}^{K+1}\right\rangle\nonumber\\
	&=\left\|\mb{t}^K-\mb{t}^*\right\|_G^2-\left\|\mb{t}^{K+1}-\mb{t}^*\right\|_G^2-\left\|\mb{t}^{K+1}-\mb{t}^{K}\right\|_G^2,
	\end{align}
We obtain that 
	\begin{align}\label{ss3_3}
	s_3=&\left\|D^{K+1}\mb{z}^{K+1}-D^\infty\mb{z}^*\right\|^2_{-2Q}+2\left\|\mb{t}^K-\mb{t}^*\right\|_G^2-2\left\|\mb{t}^{K+1}-\mb{t}^*\right\|_G^2-2\left\|\mb{t}^{K+1}-\mb{t}^{K}\right\|_G^2\nonumber\\
	&-2\left\langle G\left(\mb{t}^{K+1}-\mb{t}^K\right),\mb{t}^{*}-\mb{t}^{K+1}\right\rangle.
	\end{align}
We analyze the last two terms in Eq.~\eqref{ss3_3}:
\begin{align}\label{ss3_4}
-2\left\|\mb{t}^{K+1}-\mb{t}^{K}\right\|_G^2
\leq&-2\left\|D^{K+1}\mb{z}^{K+1}-D^K\mb{z}^{K}\right\|_{M^\top}^2,
\end{align}
where the inequality holds due to $N$-matrix norm is nonnegative, and
\begin{align}\label{ss3_5}
-&2\left\langle G\left(\mb{t}^{K+1}-\mb{t}^K\right),\mb{t}^{*}-\mb{t}^{K+1}\right\rangle\nonumber\\
=&-2\left\langle M^\top\left(D^{K+1}\mb{z}^{K+1}-D^{K}\mb{z}^{K}\right),D^\infty\mb{z}^*-D^{K+1}\mb{z}^{K+1}\right\rangle\nonumber\\
&-2\left\langle D^{K+1}\mb{z}^{K+1}-D^\infty\mb{z}^*,N^\top\left(\mb{q}^*-\mb{q}^{K+1}\right)\right\rangle\nonumber,\\
\leq&\delta\left\|D^{K+1}\mb{z}^{K+1}-D^K\mb{z}^K\right\|^2_{MM^\top}+\delta\left\|\mb{q}^*-\mb{q}^{K+1}\right\|^2_{NN^\top}\nonumber\\
&+\frac{2}{\delta}\left\|D^{K+1}\mb{z}^{K+1}-D^\infty\mb{z}^*\right\|^2,
\end{align}
for some~$\delta>0$. By substituting Eqs.~\eqref{ss3_4} and \eqref{ss3_5} into Eq.~\eqref{ss3_3}, we obtain that
\begin{align}\label{ss3_6}
s_3\leq&2\left\|\mb{t}^K-\mb{t}^*\right\|_G^2-2\left\|\mb{t}^{K+1}-\mb{t}^*\right\|_G^2+\left\|\mb{q}^*-\mb{q}^{K+1}\right\|^2_{\delta NN^\top}\nonumber\\
&+\left\|D^{K+1}\mb{z}^{K+1}-D^\infty\mb{z}^*\right\|^2_{\frac{2}{\delta}I_n-2Q}\nonumber\\
&+\left\|D^{K+1}\mb{z}^{K+1}-D^K\mb{z}^K\right\|^2_{-2M^\top+\delta MM^\top}.
\end{align}
%Our analysis also explicitly gives the range of interval, which needs the knowledge of network topology, of step-size to have DEXTRA converge to the optimal solution. Although numerical experiments show that the practical range of interval are much bigger, it is still an open question on how to get rid of the knowlege of network topology, which is a global information in a distributed setting.
{
	\small
	\bibliographystyle{IEEEbib}
	\bibliography{sample}
}
%\input{DEXTRA_TAC_UK.bbl}
%----------------------------------------------------------------------------------------

\end{document}